\documentclass[journal, twoside]{IEEEtran}
\usepackage{cite}

\usepackage{siunitx} 
\usepackage{amsmath,amssymb,amsfonts,mathtools,tensor,bm} 
\usepackage{blindtext} 
\usepackage{graphicx}
\usepackage{physics}
\usepackage{xcolor}
\usepackage{soul}
\usepackage{flushend}

\usepackage{amsthm}

\newtheoremstyle{ieeeconf}
{0pt}   
{0pt}   
{\normalfont}  
{\parindent}       
{\itshape} 
{:}         
{ } 
{\thmname{#1} \thmnumber{#2}\thmnote{ (#3)}} 
\makeatletter
\renewenvironment{proof}[1][\proofname]{\par
	\pushQED{\qed}%
	\normalfont \topsep\z@
	\trivlist
	\item[\hskip\labelsep\indent
	\itshape
	#1\@addpunct{:}]\ignorespaces
}{%
	\popQED\endtrivlist\@endpefalse
}
\makeatletter

\theoremstyle{ieeeconf}

\usepackage{relsize}

\usepackage{booktabs}
\usepackage{tabu}
\usepackage{multirow}

\usepackage{pgfplots}
\usepackage{varwidth}

\newenvironment{talign}
{\let\displaystyle\textstyle\align}
{\endalign}
\newenvironment{talign*}
{\let\displaystyle\textstyle\csname align*\endcsname}
{\endalign}

\DeclareMathOperator*{\diag}{diag}
\DeclareMathOperator*{\myw}{w}
\DeclareMathOperator*{\myinp1np2}{\MakeUppercase{\romannumeral 4}} 
\newcommand{\myi}{\vb*{v}_{\text{\normalfont{o}}}}
\DeclareMathOperator*{\mya}{a}
\DeclareMathOperator*{\myb}{b}
\DeclareMathOperator*{\myc}{c}
\DeclareMathOperator*{\myit}{it}
\DeclareMathOperator*{\myd}{d}

\DeclareMathOperator*{\myg}{g}
\DeclareMathOperator*{\mysigma}{\sigma}
\DeclareMathOperator*{\myP}{P}

\renewcommand{\d}[1]{\dot{#1}} 
\newcommand{\mat}[1]{\begin{bmatrix*}[r]#1 
\end{bmatrix*}}
\newcommand{\matc}[1]{\begin{bmatrix*}[c]#1 
\end{bmatrix*}}

\newcommand{\RN}[1]{\textup{\uppercase\expandafter{\romannumeral#1}}}
\usepackage{algorithm} 
\usepackage{algpseudocode} 
\algnewcommand{\Initialize}[1]{%
	\State \textbf{initialize} {\raggedright #1}
}
\makeatletter
\algnewcommand{\Statey}[1]{\Statex \hskip\ALG@thistlm #1}
\makeatother
\makeatletter
\algnewcommand{\Statez}[1]{\State \hskip\ALG@thistlm #1}
\makeatother
\newcounter{step}
\newenvironment{step}[1][]{\refstepcounter{step}
	\!\!\textbf{Step~\thestep. #1} \rmfamily}{}
\algdef{SE}[SUBALG]{Indent}{EndIndent}{}{\algorithmicend\ }%
\algtext*{Indent}
\algtext*{EndIndent}

\pgfplotsset{compat=1.16}

\newtheorem{theorem}{Theorem}
\newtheorem{definition}{Definition}
\newtheorem{assumption}{Assumption}
\newtheorem{problem}{Problem}
\newtheorem{remark}{Remark}

\newtheorem{proposition}{Proposition}
\newtheorem{lemma}{Lemma}

\graphicspath{{./img/}}

\begin{document}
\title{Excitation for Adaptive Optimal Control of \\Nonlinear Systems in Differential Games}

\author{Philipp~Karg*, Florian~K\"opf*, Christian~A.~Braun and~S\"oren~Hohmann
%
%
\thanks{P. Karg, F. K\"opf, C. A. Braun and S. Hohmann are with the Institute of Control Systems, Karlsruhe Institute of Technology, Karlsruhe, Germany.\\ *\textit{P. Karg and F. K\"opf are co-first authors.} (e-mail: philipp.karg@kit.edu, florian.koepf@kit.edu)}}

\maketitle

\begin{abstract}
This work focuses on the fulfillment of the Persistent Excitation (PE) condition for signals which result from transformations by means of polynomials. This is essential e.g. for the convergence of Adaptive Dynamic Programming algorithms due to commonly used polynomial function approximators.
As theoretical statements are scarce regarding the nonlinear transformation of PE signals, we propose conditions on the system state such that its transformation by polynomials is PE. To validate our theoretical statements, we develop an exemplary excitation procedure based on our conditions using a feed-forward control approach and demonstrate the effectiveness of our method in a nonzero-sum differential game.
In this setting, our approach outperforms commonly used probing noise in terms of convergence time and the degree of PE, shown by a numerical example.
\end{abstract}

\begin{IEEEkeywords}
Adaptive Dynamic Programming, Adaptive Optimal Control, Persistent Excitation.
\end{IEEEkeywords}

\section{Introduction} \label{sec:introduction}
Persistent Excitation (PE) plays a crucial role in the convergence analysis in system identification \cite{AAstrom.1966}, adaptive control \cite{Narendra.1987, Narendra.2005} and Adaptive Dynamic Programming (ADP, alias adaptive optimal control) \cite{Vamvoudakis.2010, Werbos.1992, Vamvoudakis.2011, Liu.2014, Johnson.2015, Zhang.2013}. To assess and achieve PE, the frequency content of a signal is important. For \textit{linear} algebraic and dynamic transformations, results are presented e.g. in \cite{Narendra.2005} and \cite{Boyd.1986} which show that the concept of PE is well studied in the linear setting. For example, if and only if the spectral measure of the scalar input of a linear dynamic system of order $n$ contains at least $n$ spectral lines, the system state is PE \cite[p.~255]{Narendra.2005} (see \cite{Mathelin.1990, Green.1986} for multi-dimensional inputs). However, the situation is significantly more challenging with \textit{nonlinear} transformations as additional frequencies might emerge and enhance PE or frequencies might cancel out and impede PE \cite{Lin.1999}. Unfortunately, nonlinear transformations are inevitable in the ADP context, even in the linear-quadratic setting.

Hence, the realization of excitation in ADP is still an open question and only done heuristically. In the hope of fulfilling the PE condition, ADP methods add probing noise \cite{Vamvoudakis.2011, Liu.2014, Johnson.2015, Zhao.2016, Qu.2019, Zhang.2013} or a sum of sinusoids \cite{Johnson.2015, Vamvoudakis.2016, Qu.2019} to the control input or repeatedly reset the system state \cite{Liu.2014}. 
However, none of these methods can guarantee a priori that the PE condition holds and ``no verifiable method exists to ensure PE in nonlinear systems" \cite{Johnson.2015}.
Furthermore, using probing noise may not be reasonable under application-specific requirements such as systems with low-pass characteristics\footnotemark.

Considering strategies to fulfill PE under nonlinear transformations, \cite{Adetola.2008} use a reference trajectory and an optional dither signal as excitation in nonlinear adaptive control.
In \cite{Lin.1999}, the fulfillment of PE through a chosen, exactly defined scalar reference trajectory can be checked in advance for strict-feedback systems. However, general results and procedures to design persistently excited signals are missing. Hence, approaches using alternative excitation conditions \cite{Padoan.2017, Chowdhary.2010, Kamalapurkar.2014, Kamalapurkar.2016} arose. For example, in \cite{Chowdhary.2010} (adaptive control) or \cite{Kamalapurkar.2014}, \cite{Kamalapurkar.2016} (ADP), data stored in a history stack needs to satisfy a rank condition.
Appropriate data is assumed to be available a priori (\cite[Condition~1]{Chowdhary.2010}, \cite[Assumption~2]{Kamalapurkar.2014}), but it remains unclear how to excite a system to generate such data.

\footnotetext{Human motor learning can be modeled by ADP \cite{Jiang.2014}. Thus, when modeling human-machine interaction by using game theory \cite{Flad.2017}, ADP for differential games is a promising approach when the partners do not know the control law and objective function of each other a priori. However, due to the low-pass characteristics of the human neuromuscular system \cite{Katzourakis.2014}, probing noise can prevent a meaningful learning process of the automation.}

In summary, no generic method for the analysis of PE under nonlinear transformations and no ensuing excitation procedure for ADP exists. The main contribution of this paper is given by sufficient conditions on the system state that guarantee PE under polynomial transformations which appear e.g. in ADP-based optimal control \cite{Werbos.1992, Vamvoudakis.2010} and nonzero-sum differential game scenarios \cite{Liu.2014, Vamvoudakis.2011, Johnson.2015, Qu.2019, Zhao.2016, Kamalapurkar.2014, Kopf.2018, Zhang.2013}.
These conditions are then used for an a priori calculation of excitation signals for ADP-based differential games, which generalize optimal control problems. Our sufficient conditions lead to degrees of freedom regarding amplitudes and frequencies of the system state in this exemplary excitation procedure that can be utilized in order to account for problem-specific requirements such as low-pass characteristics. Simulation results demonstrate the effectiveness and validity of our method.

In the next section, the excitation problem is stated. Our main result regarding sufficient conditions for the system state in order to guarantee PE is proposed in Section~{\ref{sec:excitation_conditions}}. An excitation procedure based on these conditions is given in Section~{\ref{sec:excitation_procedures}}. Simulation results in Section~{\ref{sec:simulation}} demonstrate the effectiveness and validity of our method before we conclude our paper.


\section{Problem Definition} \label{sec:problem}
\begin{definition}[cf. {\cite[Definition~6.2]{Narendra.2005}}] \label{def:pe_new}
	A piecewise continuously differentiable, bounded function (cf.~\cite[Definition~6.1]{Narendra.2005}) ${\vb*{\sigma}: \mathbb{R}_{\geq 0} \rightarrow \mathbb{R}^{h}}$ is called persistently excited (PE) for all $t \geq t_0$ if there exist positive constants $\alpha, \alpha_{\RN{1}}, T \in \mathbb{R}_{>0}$ such that the equivalent conditions (cf. \cite[Theorem~2.16]{Narendra.2005})
	\begin{align} 
		\int_{t}^{t+T} \vb*{\sigma}\vb*{\sigma}^{\intercal} \,\mathrm{d}\tau &\succeq \alpha\vb*{I}, \label{eq:pe_cond_1} \\
		\text{and}\quad	\frac{1}{T}\int_{t}^{t+T} \abs{\vb*{\sigma}^{\intercal}\vb*{\iota}} \,\mathrm{d}\tau &\geq \alpha_{\RN{1}}\label{eq:pe_cond_2}
	\end{align}
	hold for all $t \geq t_0$ and for any unit vector $\vb*{\iota} \in \mathbb{R}^{h}$.
\end{definition}

It is well known that an equation of the form
\begin{align} \label{eq:error_dequ}
	\dv{\tilde{\vb*{\theta}}}{t} = - \eta \vb*{\sigma}\vb*{\sigma}^{\intercal}\tilde{\vb*{\theta}}, \quad \eta > 0,
\end{align}
converges to zero exponentially with 
\begin{talign}
	\norm{\tilde{\vb*{\theta}}(kT)}_{2}&\leq\exp\left(\frac{k}{2}\ln{\rho}\right)\norm{\tilde{\vb*{\theta}}(0)}_{2}, \quad \forall k \in \mathbb{N}_{\geq 0}, \label{eq:exponential_decay} \\
	\rho&=1-\frac{2T\eta\alpha_{\RN{1}}^2}{\left(1+T\eta\sigma_{\textnormal{max}}^2\right)^2}, \quad \sigma_{\textnormal{max}}\geq\norm{\vb*{\sigma}}_{2}, \label{eq:rho}
\end{talign}
iff $\vb*{\sigma}$ is PE (cf.~\cite[Theorem~2.16]{Narendra.2005}, \cite[Technical Lemma~2]{Vamvoudakis.2010}).

In order to motivate the importance of \eqref{eq:error_dequ} in the ADP context, consider the system dynamics
\begin{align} \label{eq:dx}
	\begin{aligned}
		\dot{\vb*{x}} &= \vb*{f}(\vb*{x}) + \sum_{i=1}^{N}\vb*{g}_{i}(\vb*{x})\vb*{u}_{i} = \vb*{f}(\vb*{x}) + \vb*{g}(\vb*{x})\vb*{u}, 
	\end{aligned}
\end{align}
${\vb*{x}(0) = \vb*{x}_{0}}$, with system state ${\vb*{x}\in\mathbb{R}^n}$, control input $\vb*{u}_i\in\mathbb{R}^{p_i}$ of controller $i\in\mathcal{N}=\left\{1,2,\dots,N\right\}$, ${N\in\mathbb{N}_{\geq 1}}$, ${\vb*{f}: \mathbb{R}^{n} \rightarrow \mathbb{R}^{n}}$ and $\vb*{g}_{i}: \mathbb{R}^{n} \rightarrow \mathbb{R}^{n\times p_{i}}$. 
Here, $\vb*{g}(\vb*{x}) = \mat{\vb*{g}_{1}(\vb*{x}) & \vb*{g}_{2}(\vb*{x}) & \cdots & \vb*{g}_{N}(\vb*{x})}$, $\vb*{g}: \mathbb{R}^{n} \rightarrow \mathbb{R}^{n\times p}$, ${p = \sum_{i=1}^N p_i}$, and $\vb*{u} = \mat{\vb*{u}_{1}^\intercal & \vb*{u}_{2}^\intercal & \cdots & \vb*{u}_{N}^\intercal}^\intercal \in \mathbb{R}^{p}$. 
The initial time is $t_{0}=0$. Let $\vb*{f}(\vb*{0})=\vb*{0}$ and $\vb*{f}(\vb*{x})$, $\vb*{g}(\vb*{x})$ be Lipschitz continuous in a neighborhood $\mathcal{X} \subseteq \mathbb{R}^{n}$ of the origin and let the system~\eqref{eq:dx} be stabilizable on $\mathcal{X}$.
For admissible policies $\vb*{\mu}=\mat{\vb*{\mu}_1^\intercal & \cdots & \vb*{\mu}_N^\intercal}^\intercal$ \cite[Definition~1]{Vamvoudakis.2011} define value functions 
\begin{align} \label{eq:value_func}
	\begin{aligned}
		V_i^{\vb*{\mu}}(\vb*{x})=V_i^{\vb*{\mu}}(\vb*{x}(t)) = \int_{t}^{\infty} Q_{i}(\vb*{x}) + 	\sum_{j=1}^{N}\vb*{\mu}^{\intercal}_{j}\vb*{R}_{ij}\vb*{\mu}_{j} \,\mathrm{d}\tau,
	\end{aligned}
\end{align}
where $\vb*{\mu}_i=\vb*{\mu}_i(\vb*{x})$ is the feedback control law of player~$i$ and the $V_i$ are assumed to be continuously differentiable on $\mathcal{X}$. ${Q_{i}: \mathbb{R}^{n} \rightarrow \mathbb{R}}$ is a positive definite function \cite[p.~53]{Narendra.2005}, $\vb*{R}_{ij}\succeq \vb*{0}$, $\forall j\in\mathcal{N}, {i\neq j}$ and $\vb*{R}_{ii}=\vb*{R}_{ii}^\intercal\succ\vb*{0}$.
The aim is to find control laws $\vb*{\mu}_i^*$ where $\left\{\vb*{\mu}_1^*,\dots,\vb*{\mu}_N^*\right\}$ corresponds to a presumed to exist feedback Nash equilibrium \cite[Definition~4.1]{Basar.1998} of the $N$-player nonzero-sum differential game given by \eqref{eq:dx} and the minimization of \eqref{eq:value_func}.
 
Using an ADP formalism, a Nash strategy $\vb*{\mu}_i^*$ is deduced online with measurements of the time signals $\vb*{x}(t)$ and $\vb*{\mu}_j(t)$ and without knowing $V_j$ and $\vb*{\mu}_j(\vb*{x})$, $\forall j\neq i$. Common value function approximation 
\begin{equation} \label{eq:approx_value_i}
	V_{i}(\vb*{x}) = \vb*{\theta}^{\intercal}_{i}\vb*{\phi}_{i}(\vb*{x})
\end{equation}
with basis functions $\vb*{\phi}_{i}: \mathbb{R}^{n}\rightarrow \mathbb{R}^{h_i}$ and weights $\vb*{\theta}_{i} \in \mathbb{R}^{h_i}$ yields $\vb*{\sigma}_{i} = \d{\vb*{\phi}}_{i}(\vb*{x}(t))$ \cite{Vamvoudakis.2011}. 
If a policy iteration (see e.g. \cite[Algorithm~1]{Liu.2014} or \cite[Algorithm~1]{Song.2017}) is performed and $\vb*{\theta}_i$ is tuned using a gradient descent method during the policy evaluation step, \eqref{eq:error_dequ} results, where $\tilde{\vb*{\theta}}_i$ is the weight error \cite{Vamvoudakis.2010}. Thus, convergence results iff $\vb*{\sigma}_i$, $i\in\mathcal{N}$ is PE\footnote{Exact convergence and relation~\eqref{eq:error_dequ} follow if the value functions can be approximated exactly. Otherwise, convergence to a residual set results \cite[Technical Lemma~2]{Vamvoudakis.2010}. Throughout the paper we assume exact approximations since we focus on the excitation problem.}. The policy improvement step updates the policy according to
\begin{equation} \label{eq:approx_mu_i}
	{\vb*{\mu}_{i}(\vb*{x}) = - \frac{1}{2}\vb*{R}^{-1}_{ii}\vb*{g}^{\intercal}_{i}(\vb*{x})\left(\pdv{\vb*{\phi}_{i}(\vb*{x})}{\vb*{x}} \right)^{\intercal}\vb*{\theta}_{i}}
\end{equation}
(cf.~ \cite{Song.2017}).

\begin{assumption}\label{asm:phi_h}
	The $h_i$ elements of each basis function vector $\vb*{\phi}_{i}(\vb*{x})$ ($\forall i \in \mathcal{N}$) are chosen as
	\begin{align} \label{eq:assumption_phi}
		\begin{aligned}
			\phi_{i,\bar{h}}(\vb*{x}) = &\bar{f}_{i,\bar{h}} \prod_{j=1}^{n} x^{f_{i,j,\bar{h}}}_{j}, \quad \forall \bar{h}\in\{1,\dots,h_i\},
		\end{aligned}
	\end{align}
	$\bar{f}_{i,\bar{h}} \in \mathbb{R}_{\neq 0}, f_{i,j,\bar{h}} \in \mathbb{N}_{0}$, where at least for one $j$ $f_{i,j,\bar{h}} \neq 0$ holds and $\bar{f}_{i,\bar{h}}$ is finite.
\end{assumption}

\begin{problem} \label{problem:PE}
	Let  $\vb*{\phi}_{i}(\vb*{x})$ ($\forall i \in \mathcal{N}$) be consistent with Assumption~\ref{asm:phi_h}. Find $\vb*{x}(t)$ such that $\vb*{\sigma}_i=\d{\vb*{\phi}}_{i}(\vb*{x}(t))$ is PE $\forall i \in \mathcal{N}$.
\end{problem}


\section{Sufficient Conditions for PE} \label{sec:excitation_conditions}
Since $\vb*{\phi}_{i}(\vb*{x})$ is usually nonlinear, it cannot be handled by typical statements in the PE literature. For ease of notation, the index $i$ is omitted first and we examine, without loss of generality (w.l.o.g.), $\vb*{\phi}(\vb*{x}): \mathbb{R}^{n}\rightarrow \mathbb{R}^{h}$. We conclude with the consideration of all $N$ players. In order to find trajectories $\vb*{x}(t)$ that solve Problem~\ref{problem:PE}, we start by choosing a general structure for desired system states.

\begin{assumption}\label{asm:xd_o}
	The elements of the desired system states ${\vb*{x}_{\myd}\in\mathbb{R}^n}$ are chosen according to
	\begin{align} \label{eq:assumption_x_d}
		\begin{aligned}
			x_{\myd,o}	= &\sum_{j=1}^{m} g_{j,o}\sin(\omega_{j}t) + \sum_{j=1}^{m} \bar{g}_{j,o}\cos(\omega_{j}t),
		\end{aligned}
	\end{align}
	$g_{j,o},\,\bar{g}_{j,o} \in \mathbb{R}, \quad \forall o \in \{1,\dots,n\}$, where at least for one $j$ $g_{j,o}\neq 0$ or $\bar{g}_{j,o}\neq 0$ holds, $g_{j,o}, \bar{g}_{j,o}$ are finite, ${m\geq 1}$ and $\vb*{\omega} = \matc{\omega_{1} & \dots & \omega_{m}}^{\intercal} \in \mathbb{R}^{m}$ are frequency variables.
\end{assumption}

\begin{lemma} \label{lemma:trig_simplification}
	With Assumption~\ref{asm:phi_h} and \ref{asm:xd_o}, $\vb*{\phi}(\vb*{x}_{\myd})$ is given by
	\begin{align} \label{eq:phi_x_d_short}
		\phi_{\bar{h}}(\vb*{x}_{\myd}) = \sum_{l=1}^{L_{\bar{h}}} a_{l,\bar{h}} \sin({\myw}_{l,\bar{h}} t) + \sum_{k=1}^{K_{\bar{h}}} c_{k,\bar{h}} \cos(\bar{\myw}_{k,\bar{h}} t) + e_{\bar{h}}
	\end{align}
	$\forall \bar{h} \in \{1,...,h\}$, where $a_{l,\bar{h}}, c_{k,\bar{h}} \in \mathbb{R}_{\neq 0}$, $e_{\bar{h}} \in \mathbb{R}$ and 
	\begin{talign}
		{\myw}_{l,\bar{h}} &= \sum_{j=1}^{m} b_{j,l,\bar{h}} \omega_{j}, \\
		\bar{\myw}_{k,\bar{h}} &= \sum_{j=1}^{m} d_{j,k,\bar{h}} \omega_{j}
	\end{talign}
	with $b_{j,l,\bar{h}}, d_{j,k,\bar{h}} \in \mathbb{R}$. Furthermore, for the upper sum limits $L_{\bar{h}}, K_{\bar{h}} \in \mathbb{N}_{\geq 0}$, $L_{\bar{h}}+K_{\bar{h}}\geq1$ ($\forall \bar{h} \in \{1,\dots,h\}$) holds.
\end{lemma}
\begin{proof}
	The proof is given in Appendix~\ref{app:proof_lemma_trig_simplification}.
\end{proof}

The frequencies ${\myw}_{l,\bar{h}}$ and $\bar{\myw}_{k,\bar{h}}$ in \eqref{eq:phi_x_d_short} are analyzed in order to derive excluding conditions on $\vb*{\omega}$ such that $\d{\vb*{\phi}}(\vb*{x}_{\myd})$ is PE. 

\begin{proposition} \label{proposition:def_omega}
	Let $\vb*{\phi}(\vb*{x}_{\myd})$ be as in Lemma~\ref{lemma:trig_simplification}. 
	Then, ${\vb*{\sigma} = \d{\vb*{\phi}}(\vb*{x}_{\myd})}$ is PE ${\forall \vb*{\omega} \in \Omega}$, if a non-empty set $\Omega \subseteq \mathbb{R}^{m}$ exists such that $\forall \vb*{\omega} \in \Omega$ and every non-zero constant $\vb*{\alpha} \in \mathbb{R}^{h}$
	\begin{align} \label{eq:alpha_sig}
		\vb*{\alpha}^{\intercal}\vb*{\sigma}(\vb*{x}_{\myd}) &= \vb*{\alpha}^{\intercal}\d{\vb*{\phi}}(\vb*{x}_{\myd}) \\
		&= \sum_{l=1}^{L^{(\alpha)}} a^{(\alpha)}_{l} \cos({\myw}^{(\alpha)}_{l}t) 
		+ \sum_{k=1}^{K^{(\alpha)}} c^{(\alpha)}_{k}\sin(\bar{\myw}^{(\alpha)}_{k}t)\nonumber
	\end{align}
	holds, where\footnote{The superscript $(\alpha)$ indicates a dependency on $\vb*{\alpha}$.} ${L^{(\alpha)} + K^{(\alpha)}\geq1}$ and
	\begin{equation*}
		\begin{aligned}[c]
			a^{(\alpha)}_{l} &\neq 0,\\
			{\myw}^{(\alpha)}_{l} &\neq 0,\\
			{\myw}^{(\alpha)}_{l_{1}} &\neq {\myw}^{(\alpha)}_{l_{2}},\\
			\bar{\myw}^{(\alpha)}_{k_{1}} &\neq \bar{\myw}^{(\alpha)}_{k_{2}},
		\end{aligned}
		\qquad
		\begin{aligned}[c]
			c^{(\alpha)}_{k} &\neq 0,\\
			\bar{\myw}^{(\alpha)}_{k} &\neq 0,\\
			{\myw}^{(\alpha)}_{l_{1}} &\neq -{\myw}^{(\alpha)}_{l_{2}},\\
			\bar{\myw}^{(\alpha)}_{k_{1}} &\neq -\bar{\myw}^{(\alpha)}_{k_{2}},\\
		\end{aligned}
	\end{equation*}
	$\forall l_{1/2} \in \left\{1,\dots,L^{(\alpha)}\right\}$, $l_{1}\neq l_{2}$ and ${\forall k_{1/2} \in \left\{1,\dots,K^{(\alpha)}\right\}}$, ${k_{1}\neq k_{2}}$. 
\end{proposition}
\begin{proof}
	The signal in \eqref{eq:alpha_sig} has at least one spectral line under its associated conditions due to the linear independence of its summands. Hence, it is PE \cite[Sublemma~6.1]{Narendra.2005}. Since this holds for every constant non-zero vector $\vb*{\alpha} \in \mathbb{R}^{h}$, the PE of $\vb*{\sigma}(\vb*{x}_{\myd}) = \d{\vb*{\phi}}(\vb*{x}_{\myd})$ follows from \cite[Lemma~6.2]{Narendra.2005}.
\end{proof}

In line with Proposition~\ref{proposition:def_omega}, it is sufficient to find a non-empty set $\Omega$ containing real vectors $\vb*{\omega}$ such that \eqref{eq:alpha_sig} holds. To calculate such a set $\Omega$, we use the frequencies in $\sigma_{\bar{h}} = \d{\phi}_{\bar{h}}(\vb*{x}_{\myd})$ $(\forall \bar{h} \in \{1,...,h\})$ as starting point which can be directly deduced from \eqref{eq:phi_x_d_short}. First, we define two auxiliary sets $\Omega^{(1)}$ and $\Omega^{(2)}$. Then, our main theorem states that ${\Omega = \Omega^{(1)} \cap \Omega^{(2)}}$ is a set in the sense of Proposition~\ref{proposition:def_omega}.

\begin{definition}[$\Omega^{(1)}$: Uniqueness of the Frequencies in Each Element $\phi_{\bar{h}}(\vb*{x}_{\myd})$ ($\forall \bar{h} \in \{1,...,h\}$)]\label{def:Omega1_def}
	Let
	\begin{talign} \label{eq:omega_1}
		\begin{aligned}
			\Omega^{(1)} \! = &\biggl\{ \vb*{\omega}: \bigwedge\limits_{\bar{h}=1}^{h} \biggl( \bigwedge\limits_{l_{1}=1}^{L_{\bar{h}}-1} \bigwedge\limits_{l_{2}>l_{1}}^{L_{\bar{h}}} \left( {\myw}_{l_{1},\bar{h}} \! \neq \! {\myw}_{l_{2},\bar{h}} \wedge {\myw}_{l_{1},\bar{h}} \! \neq \! -{\myw}_{l_{2},\bar{h}} \right) \\  
			&\hphantom{\biggl\{}\wedge \! \bigwedge\limits_{k_{1}=1}^{K_{\bar{h}}-1} \bigwedge\limits_{k_{2}>k_{1}}^{K_{\bar{h}}} \left( \bar{\myw}_{k_{1},\bar{h}} \! \neq \! \bar{\myw}_{k_{2},\bar{h}} \wedge \bar{\myw}_{k_{1},\bar{h}} \! \neq \! -\bar{\myw}_{k_{2},\bar{h}} \right) \biggl) \biggl\}.
		\end{aligned}
	\end{talign}
\end{definition}

\begin{lemma} \label{lemma:simpl_omega_1}
	$\Omega^{(1)}$ in Definition~\ref{def:Omega1_def} can be rewritten as
	\begin{talign} \label{eq:omega_1_red}
		\Omega^{(1)} = \left\{ \vb*{\omega}: \bigwedge\limits_{z=1}^{Z_{1}} 	\vb*{c}^{\intercal}_{z}\vb*{\omega} \neq 0 \right\}, \quad Z_1 \in \mathbb{N}_{\geq 0}.
	\end{talign}
\end{lemma}
\begin{proof}
	The complementary set $\bar{\Omega}^{(1)}$ is a disjunction of linear algebraic equations of a form like ${\myw}_{l_{1},\bar{h}} = {\myw}_{l_{2},\bar{h}}$ (De Morgan's laws). Denote their matrix equivalents as $\bar{\vb*{c}}^{\intercal}_{z}\vb*{\omega}=0$ ($\bar{\vb*{c}}_{z} \in \mathbb{R}^{m}$) with running index $z$. Bring $\bar{\vb*{c}}^{\intercal}_{z}$ to reduced row echelon form $\vb*{c}^{\intercal}_{z}$, remove identical $\vb*{c}^{\intercal}_{z}$\footnote{Repeatedly occurring conditions do not change the set $\bar{\Omega}^{(1)}$.} and employ De Morgan's laws again to obtain \eqref{eq:omega_1_red}.  
\end{proof}

\begin{definition}[Representation of Frequencies]\label{def:PT}
	Let
	\begin{equation}
		W_{\bar{h}} \coloneqq\left\{ {\myw}_{1,\bar{h}}, \dots, {\myw}_{L_{\bar{h}},\bar{h}}, 
		\bar{\myw}_{1,\bar{h}}, \dots, \bar{\myw}_{K_{\bar{h}},\bar{h}} \right\},
	\end{equation}
	$\forall\bar{h}\in\{1,\dots,h\}$ (cf.~\eqref{eq:phi_x_d_short}). Then, ${N_{\myP} = \prod_{\bar{h}=1}^{h}\left(L_{\bar{h}}+K_{\bar{h}}\right)}$ different tuples $\mathcal{P}_i\coloneqq\left(\zeta_1^{(i)},\dots ,\zeta_h^{(i)}\right)$, $i=1,\dots,N_{\myP}$, $\zeta_{\bar{h}}^{(i)}\in W_{\bar{h}}$ ($\bar{h}\in\{1,\dots,h\}$), are defined. 
	Let $\vb*{P}$ be a matrix of dimension $h \times N_{\myP}$ whose $N_{\myP}$ columns contain the entries of $\mathcal{P}_1, \dots, \mathcal{P}_{N_{\myP}}$. The corresponding trigonometric function type is encoded in a matrix $\vb*{T}\in \mathbb{R}^{h \times N_{\myP}}$. Its $(\bar{h},n_{\myP})$-th element is set to zero ($T_{\bar{h},n_{\myP}}=0$) if the $(\bar{h},n_{\myP})$-th element $P_{\bar{h},n_{\myP}}$ of $\vb*{P}$ is the argument of a sine function (cf.~\eqref{eq:phi_x_d_short}). Otherwise, $T_{\bar{h},n_{\myP}}=1$ holds.
\end{definition}

\begin{definition}[$\Omega^{(2)}$: Conditions on the Frequencies in $\vb*{P}$] \label{def:Omega_2_def}
	With $\vb*{P}$ and $\vb*{T}$ as given in Definition~\ref{def:PT}, define
	\begin{talign} \label{eq:omega_2}
		\begin{aligned}
			\Omega^{(2)} \! = &\biggl\{ \vb*{\omega}:\! \bigvee\limits_{n_{\myP}=1}^{N_{\myP}} \biggl( \underbrace{\bigwedge\limits_{\bar{h}=1}^{h} P_{\bar{h},n_{\myP}} \! \neq \! 0}_{\RN{1}}  \\[-0.2cm]
			&\hphantom{\biggl\{}\!\!\!\!\!\!\!\wedge \! \underbrace{\bigwedge\limits_{\bar{h}_{1}=1}^{h-1} ~\bigwedge\limits_{\bar{h}_{2}\in\bar{H}} \left( P_{\bar{h}_{1},n_{\myP}} \! \neq \! P_{\bar{h}_{2},n_{\myP}} \wedge P_{\bar{h}_{1},n_{\myP}} \! \neq \! -P_{\bar{h}_{2},n_{\myP}} \right)}_{\RN{2}} \\[-0.2cm]
			&\hphantom{\biggl\{}\!\!\!\!\!\!\!\wedge \! \underbrace{\bigwedge\limits_{\substack{\bar{h}_{1/2}=1\\\bar{h}_{1}\neq\bar{h}_{2}}}^{h} ~\bigwedge\limits_{\bar{l}\in\bar{L}} \left( P_{\bar{h}_{1},n_{\myP}} \! \neq \! {\myw}_{\bar{l},\bar{h}_{2}} \wedge P_{\bar{h}_{1},n_{\myP}} \! \neq \! -{\myw}_{\bar{l},\bar{h}_{2}} \right)}_{\RN{3}_{\mya}} \\[-0.2cm]
			&\hphantom{\biggl\{}\!\!\!\!\!\!\!\wedge \! \underbrace{\bigwedge\limits_{\substack{\bar{h}_{1/2}=1\\\bar{h}_{1}\neq\bar{h}_{2}}}^{h} ~\bigwedge\limits_{\bar{k}\in\bar{K}} \left( P_{\bar{h}_{1},n_{\myP}} \! \neq \! \bar{\myw}_{\bar{k},\bar{h}_{2}} \wedge P_{\bar{h}_{1},n_{\myP}} \! \neq \! -\bar{\myw}_{\bar{k},\bar{h}_{2}} \right)}_{\RN{3}_{\myb}} \biggl)\! \biggl\}, 
		\end{aligned}
	\end{talign}
	where the index sets are given by
	\begin{talign} \label{eq:omega_2_index}
		\begin{aligned}
			\bar{H} &= \bar{H}(\bar{h}_{1},n_{\myP})\\
			&= \left\{ \bar{h}_{2} \in \{1,\dots,h\}: \bar{h}_{2}>\bar{h}_{1} \wedge T_{\bar{h}_{1},n_{\myP}} = T_{\bar{h}_{2},n_{\myP}} \right\},\\
			\bar{L} &= \bar{L}(\bar{h}_{1},\bar{h}_{2},n_{\myP})\\
			&= \left\{ \bar{l} \in \{1,\dots,L_{\bar{h}_{2}}\}: {\myw}_{\bar{l},\bar{h}_{2}} \neq P_{\bar{h}_{2},n_{\myP}} \wedge T_{\bar{h}_{1},n_{\myP}}=0\right\}, \\
			\bar{K} &= \bar{K}(\bar{h}_{1},\bar{h}_{2},n_{\myP}) \\
			&= \left\{ \bar{k} \in \{1,\dots,K_{\bar{h}_{2}}\}: \bar{\myw}_{\bar{k},\bar{h}_{2}} \neq P_{\bar{h}_{2},n_{\myP}} \wedge T_{\bar{h}_{1},n_{\myP}}=1\right\}.
		\end{aligned}
	\end{talign}
\end{definition}

The complementary set $\bar{\Omega}^{(2)}$ can be simplified by means of Algorithm~\ref{alg:algorithm} yielding
\begin{talign} \label{eq:omega_2_komp}
	\begin{aligned}
		\bar{\Omega}^{(2)} = &	
		\biggl\{ \vb*{\omega}:\! \bigwedge\limits_{n_{\myP}=1}^{N_{\myP}} \biggl( \bigvee\limits_{\bar{h}=1}^{h} P_{\bar{h},n_{\myP}} \! = \! 0  \\
		&\hphantom{\Biggl\{}\!\!\!\!\!\!\!\vee \! \bigvee\limits_{\bar{h}_{1}=1}^{h-1} ~\bigvee\limits_{\bar{h}_{2}\in\bar{H}} \left( P_{\bar{h}_{1},n_{\myP}} \! = \! P_{\bar{h}_{2},n_{\myP}} \vee P_{\bar{h}_{1},n_{\myP}} \! = \! -P_{\bar{h}_{2},n_{\myP}} \right) \\
		&\hphantom{\biggl\{}\!\!\!\!\!\!\!\vee \! \bigvee\limits_{\substack{\bar{h}_{1/2}=1\\\bar{h}_{1}=\bar{h}_{2}}}^{h} ~\bigvee\limits_{\bar{l}\in\bar{L}} \left( P_{\bar{h}_{1},n_{\myP}} \! = \! {\myw}_{\bar{l},\bar{h}_{2}} \vee P_{\bar{h}_{1},n_{\myP}} \! = \! -{\myw}_{\bar{l},\bar{h}_{2}} \right) \\
		&\hphantom{\biggl\{}\!\!\!\!\!\!\!\vee \! \bigvee\limits_{\substack{\bar{h}_{1/2}=1\\\bar{h}_{1}=\bar{h}_{2}}}^{h} ~\bigvee\limits_{\bar{k}\in\bar{K}} \left( P_{\bar{h}_{1},n_{\myP}} \! = \! \bar{\myw}_{\bar{k},\bar{h}_{2}} \vee P_{\bar{h}_{1},n_{\myP}} \! = \! -\bar{\myw}_{\bar{k},\bar{h}_{2}} \right) \biggl)\! \biggl\} \\
		= &\Big\{ \vb*{\omega}: \underbrace{\left. \left(\bigvee (\cdot)=(\cdot) \right)\right|_{n_{\myP}=1} \wedge \left.\left( \bigvee (\cdot)=(\cdot) \right)\right|_{n_{\myP}=2}}_{\myinp1np2}  \\ 
		&\hphantom{\biggl\{}   \wedge \bigwedge\limits_{n_{\myP}=3}^{N_{\myP}} \left( \bigvee\limits (\cdot) = (\cdot) \right) \Big\} \\
		= &\left\{ \vb*{\omega}: \left( \bigvee\limits_{z} \vb*{C}_{z} \vb*{\omega} = \vb*{0} \right) \wedge \bigwedge\limits_{n_{\myP}=3}^{N_{\myP}} \left( \bigvee\limits (\cdot) = (\cdot) \right) \right\} \\
		= &\left\{ \vb*{\omega}: \bigvee\limits_{z=1}^{Z_{2}} \vb*{C}_{z} \vb*{\omega} = \vb*{0} \right\}.
	\end{aligned}
\end{talign}

\begin{algorithm}
	\caption{Simplification of $\bar{\Omega}^{(2)}$}
	\label{alg:algorithm}
	\begin{algorithmic}[1]
		\renewcommand{\algorithmicrequire}{\textbf{Input:}}
		\renewcommand{\algorithmicensure}{\textbf{Output:}}
		
		\State Initialize $k_{\myit} = 1$
		
		\State\begin{step}\label{step:expand_conjunction}
			Build homogeneous systems of linear algebraic equations by expanding the conjunction between the conditions subscripted by $n_{\myP}=k_{\myit}$ and $n_{\myP}=k_{\myit}+1$. If $k_{\myit}+1> N_{\myP}$, this conjunction is omitted.
		\end{step}
		\State\begin{step}\label{step:solve_system}
			Solve the systems of equations for $\vb*{\omega}$. Denote freely chosen variables as $\omega_{j}=\omega_{j}$ and constitute the results as conditions of the form $\bar{\vb*{C}}_{z}\vb*{\omega}=\vb*{0}$ ($\bar{\vb*{C}}_{z} \in \mathbb{R}^{m\times m}$).
		\end{step}	
		
		\If {$\omega_{j}=\omega_{j}$ results $\forall j \in \{1,\dots,m\}$}
		
		\If {$k_{\myit}\geq N_{\myP}-1$}
		\State\Return{$\bar{\Omega}^{(2)}=\mathbb{R}^m$}
		\Else
		\State set $k_{\myit}=k_{\myit}+2$ and go back to Step~\ref{step:expand_conjunction}
		\EndIf
		
		\Else
		\State set $k_{\myit}=k_{\myit}+1$ and go to Step~\ref{step:transform_matrices}
		\EndIf
		
		\State\begin{step}\label{step:transform_matrices}
			Transform $\bar{\vb*{C}}_{z}$ into reduced row echelon form and delete zero rows in the resulting matrices. This yields matrices $\vb*{C}_{z}$ ($\vb*{C}_{z} \in \mathbb{R}^{\bar{m} \times m}, \bar{m} < m$).
		\end{step}
		\State\begin{step}\label{step:remove_repeated}
			Remove repeatedly occurring conditions (identical $\vb*{C}_{z}$) and represent the simplified conditions of the expanded conjunction as $\bigvee_{z} \vb*{C}_{z} \vb*{\omega} = \vb*{0}$. 
		\end{step}
		\State\begin{step}\label{step:while}	
			
			\While{$k_{\myit}\leq N_{\myP}-1$}
			\State \parbox[t]{\dimexpr\columnwidth-\leftmargin-\labelsep-\labelwidth}{Build homogeneous systems of linear algebraic equations by expanding the conjunction between the result of Step~\ref{step:remove_repeated} and the conditions subscripted by ${n_{\myP}=k_{\myit}+1}$.
				Apply {Steps~\ref{step:solve_system}--\ref{step:remove_repeated}} to these systems.}
			\EndWhile
			\State\Return{$\bar{\Omega}^{(2)}$}
			
		\end{step}
	\end{algorithmic} 
\end{algorithm}

\begin{lemma} \label{lemma:simpl_omega_2}
	If Algorithm~\ref{alg:algorithm} returns $\bar{\Omega}^{(2)} = \mathbb{R}^m$, $\Omega^{(2)} = \emptyset$ follows. Otherwise, \eqref{eq:omega_2} can be rewritten as
	\begin{align} \label{eq:omega_2_red}
		\Omega^{(2)} = \left\{ \vb*{\omega}: \bigwedge\limits_{z=1}^{Z_{2}} \vb*{C}_{z}\vb*{\omega}\neq\vb*{0} \right\}, \quad Z_2 \in \mathbb{N}_{\geq 1}.
	\end{align}
\end{lemma}
\begin{proof}
	One needs to show the validity of Steps~\ref{step:transform_matrices} and \ref{step:remove_repeated} of Algorithm~\ref{alg:algorithm} as well as the correctness of the conclusions in case of $\omega_{j}=\omega_{j}$ ($\forall j \in \{1,\dots,m\}$) in Step~\ref{step:solve_system}.
	
	First, consider Step~\ref{step:transform_matrices}. Transforming $\bar{\vb*{C}}_{z}$ into reduced row echelon form does not change the solution set of $\bar{\vb*{C}}_{z}\vb*{\omega}=\vb*{0}$. Since zero rows in $\bar{\vb*{C}}_{z}$ correspond with $\omega_{j}=\omega_{j}$, they impose conditions which cannot be fulfilled in $\Omega^{(2)}$ and hence, other rows need to be fulfilled later. Therefore, zero rows in $\bar{\vb*{C}}_{z}$ can be deleted. 
	Regarding Step~\ref{step:remove_repeated}, repeatedly occurring conditions can be omitted without loss of information. Now, w.l.o.g., examine the case when $\omega_{j}=\omega_{j}$ occurs for all $j \in \{1,\dots,m\}$ in part~$\myinp1np2$ of \eqref{eq:omega_2_komp} in Step~\ref{step:solve_system}. Then, $\myinp1np2$ holds true for any $\vb*{\omega}\in \mathbb{R}^{m}$, independent of other possible conditions of the form $\bar{\vb*{C}}_{z}\vb*{\omega}=\vb*{0}$ derived in part~$\myinp1np2$. Hence, these other conditions can be omitted and the Algorithm has to be started with Step~\ref{step:expand_conjunction} and the equations subscripted by $n_{\myP}=3$ and $n_{\myP}=4$. Since this explanation is also applicable for the repetitions in Step~\ref{step:while} and the appearance of this special case in the last simplification step ($k_{\myit}\geq N_{\myP}-1$) results in $\bar{\Omega}^{(2)}=\mathbb{R}^{m}$ ($\Omega^{(2)}=\emptyset$) and is defined as well, the proof is complete. 
\end{proof}

\begin{lemma} \label{lemma:intersection}
	The intersection of $\Omega^{(1)}$~\eqref{eq:omega_1_red} and $\Omega^{(2)}$~\eqref{eq:omega_2_red} yields
	\begin{align} \label{eq:omega_red}
		\begin{aligned}
			\Omega &= \Omega^{(1)} \cap \Omega^{(2)} = \left\{ \vb*{\omega}: \bigwedge\limits_{z=1}^{Z_{1} + Z_{2}} \vb*{C}_{z}\vb*{\omega}\neq\vb*{0} \right\} \\
			&= \left\{ \vb*{\omega}: \bigwedge\limits_{z=1}^{Z} \vb*{C}_{z}\vb*{\omega}\neq\vb*{0} \right\},   \quad Z \in \mathbb{N}_{\geq 1},
		\end{aligned}
	\end{align}
	$Z\leq Z_{1}+Z_{2}$. The last row follows by removing each matrix $\vb*{C}_{z_{2}}$ that contains another matrix $\vb*{C}_{z_{1}}$ ($z_{1}\neq z_{2}$) as submatrix.
\end{lemma}
\begin{proof}
	$\Omega^{(1)} \cap \Omega^{(2)}$ yields $Z_{1}+Z_{2}$ conditions. Due to ${\vb*{C}_{z_{1}}\vb*{\omega} \neq \vb*{0} \Rightarrow \vb*{C}_{z_{2}}\vb*{\omega} \neq \vb*{0}}$, $\vb*{C}_{z_{2}}$ can be omitted. 
\end{proof}

Now, our main theorem shows that the set $\Omega$~\eqref{eq:omega_red} is a set in the sense of Proposition~\ref{proposition:def_omega}.
\begin{theorem} \label{theorem:frequ_cond}
	Let Assumption~\ref{asm:phi_h} and \ref{asm:xd_o} hold and $\Omega$ be defined as in Lemma~\ref{lemma:intersection}. If $\Omega \neq \emptyset$, any vector $\vb*{\omega} \in \Omega$ ensures that $\vb*{\sigma} = \d{\vb*{\phi}}(\vb*{x}_{\myd})$ is PE.   
\end{theorem}
\begin{proof}
	Lemma~\ref{lemma:trig_simplification} enables the definition of the auxiliary sets $\Omega^{(1)}$ and $\Omega^{(2)}$ in their initial forms \eqref{eq:omega_1} and \eqref{eq:omega_2}. Since Lemmata~\ref{lemma:simpl_omega_1}, \ref{lemma:simpl_omega_2} and \ref{lemma:intersection} solely apply equivalent transformations to these sets, we need to prove that the conditions in \eqref{eq:omega_1} and \eqref{eq:omega_2} are sufficient to fulfill \eqref{eq:alpha_sig}. 
	First, the time derivative of \eqref{eq:phi_x_d_short} yields 
	\begin{equation}	
		\sigma_{\bar{h}} = \sum_{l=1}^{L_{\bar{h}}} a_{l,\bar{h}}{\myw}_{l,\bar{h}} 	\cos({\myw}_{l,\bar{h}} t) - \sum_{k=1}^{K_{\bar{h}}} c_{k,\bar{h}}\bar{\myw}_{k,\bar{h}} \sin(\bar{\myw}_{k,\bar{h}} t),
	\end{equation} 
	${\forall \bar{h} \in \{1,\dots,h\}}$. With ${\vb*{\alpha} \in \mathbb{R}^{h}}$,
	\begin{align}
		\vb*{\alpha}^{\intercal}\vb*{\sigma} &= \sum_{\bar{h}=1}^{h} \!\left( \alpha_{\bar{h}} \!\left( \sum_{l=1}^{L_{\bar{h}}} a_{l,\bar{h}}{\myw}_{l,\bar{h}} \cos({\myw}_{l,\bar{h}} t)\right.\right. \nonumber \\ & \phantom{-\alpha\sum_{\bar{h}=1}^{h} \!\left( \alpha_{\bar{h}} \!\left(\right.\right.} \left.\left.
		 - \sum_{k=1}^{K_{\bar{h}}} c_{k,\bar{h}}\bar{\myw}_{k,\bar{h}} \sin(\bar{\myw}_{k,\bar{h}} t) \right)\right)
	\end{align}	
	results. Now, w.l.o.g., assume that the exemplarily chosen column $\mathcal{P} = \left(\zeta_1,\dots ,\zeta_h\right)$ of $\vb*{P}$ with $\zeta_{\bar{h}}={\myw}_{1,\bar{h}}$ ($\forall \bar{h} \in \{1,\dots,h-1\}$) and $\zeta_{h} = \bar{\myw}_{1,h}$ fulfills conditions $\RN{1}$, $\RN{2}$ and $\RN{3}$ in \eqref{eq:omega_2}. Starting from $\vb*{\alpha}^{\intercal}\vb*{\sigma}$, the trigonometric functions, which have the corresponding frequencies as arguments, are separated and the remaining terms inside the brackets $\alpha_{\bar{h}}(\cdot)$ ($\forall \bar{h} \in \{1,\dots,h\}$) are summarized by $\varepsilon_{\bar{h}}$:
	\begin{align} \label{eq:bew_omega_3}
		\begin{aligned}
			\vb*{\alpha}^{\intercal}\vb*{\sigma} &= \alpha_{1}\left( a_{1,1}{\myw}_{1,1}\cos({\myw}_{1,1} t) + \varepsilon_{1} \right) + \dots \\
			&\hphantom{=} + \alpha_{h-1}\left( a_{1,h-1}{\myw}_{1,h-1}\cos({\myw}_{1,h-1} t) + \varepsilon_{h-1} \right) \\
			&\hphantom{=} + \alpha_{h}\left( -c_{1,h}\bar{\myw}_{1,h}\sin(\bar{\myw}_{1,h} t) + \varepsilon_{h} \right) \\
			&= \sum_{\bar{h}=1}^{h-1} \alpha_{\bar{h}} a_{1,\bar{h}} {{\myw}}_{1,\bar{h}} \cos({\myw}_{1,\bar{h}} t)  \\[-0.3cm]
			&\hphantom{=} - \alpha_{h} c_{1,h} \bar{\myw}_{1,h} \sin(\bar{\myw}_{1,h} t) + \sum_{\bar{h}=1}^{h} \alpha_{\bar{h}} \varepsilon_{\bar{h}}.
		\end{aligned}
	\end{align}

	First, the conditions introduced by $\Omega^{(1)}$~\eqref{eq:omega_1} ensure that the trigonometric function separated inside one bracket $\alpha_{\bar{h}}(\cdot)$ in \eqref{eq:bew_omega_3} is not eliminated by trigonometric terms inside the same bracket $\alpha_{\bar{h}}(\cdot)$. This results from the claimed uniqueness of the frequencies in Definition~\ref{def:Omega1_def} which leads to a linear independence of all trigonometric terms inside $\alpha_{\bar{h}}(\cdot)$. Moreover, the conditions in $\Omega^{(1)}$ guarantee that in the element $\phi_{\bar{h}}(\vb*{x}_{\myd})$ $L_{\bar{h}}$ sine and $K_{\bar{h}}$ cosine functions are existent. This assumption is required to establish $\vb*{P}$ (cf.~Definition~\ref{def:PT}).
	
	The conditions labeled by $\RN{3}$ in $\Omega^{(2)}$~\eqref{eq:omega_2} ensure that the trigonometric function separated inside the brackets $\alpha_{\bar{h}_{1}}(\cdot)$ in \eqref{eq:bew_omega_3} will not be eliminated by trigonometric functions in $\varepsilon_{\bar{h}_{2}}$ inside $\alpha_{\bar{h}_{2}}(\cdot)$ ($\forall \bar{h}_{1},\bar{h}_{2} \in \{1,\dots,h\}, \bar{h}_{1} \neq \bar{h}_{2}$), which again holds because of the uniqueness of the frequencies.
	Overall, due to the conditions in $\Omega^{(1)}$ and those denoted by $\RN{3}$ in $\Omega^{(2)}$, the separated trigonometric terms in \eqref{eq:bew_omega_3} cannot be eliminated by functions in $\varepsilon_{\bar{h}}$ ($\forall \bar{h} \in \{1,\dots,h\}$), independent of the concrete values of $\alpha_{\bar{h}}$, $a_{l,\bar{h}}$ and $c_{k,\bar{h}}$. Thus, it is sufficient to show that under $\varepsilon_{\bar{h}}=0$ ($\forall \bar{h} \in \{1,\dots,h\}$) the conditions labeled by $\RN{1}$ and $\RN{2}$ in $\Omega^{(2)}$~\eqref{eq:omega_2} cause a signal for $\vb*{\alpha}^{\intercal}\vb*{\sigma}$ according to \eqref{eq:alpha_sig}. If $\varepsilon_{\bar{h}} \neq 0$ holds, existing trigonometric terms cannot cancel out, only more can be added. 
	
	Now, suppose $\varepsilon_{\bar{h}} = 0$ ($\forall \bar{h} \in \{1,\dots,h\}$). Then, from the conditions $\RN{1}$ in $\Omega^{(2)}$ ${\myw}_{1,\bar{h}} \neq 0$ and ${a_{1,\bar{h}}{\myw}_{1,\bar{h}} \neq 0}$ ${(\forall \bar{h} \in \{1,\dots,h-1\})}$ follows as well as $\bar{\myw}_{1,h} \neq 0$ and ${c_{1,h}\bar{\myw}_{1,h} \neq 0}$. Under the assumption $\varepsilon_{\bar{h}} = 0$ ${(\forall \bar{h} \in \{1,\dots,h\})}$, these conditions are also necessary to satisfy \eqref{eq:alpha_sig}. If one frequency is equal to zero, an $\vb*{\alpha}\neq\vb*{0}$ exists with $\vb*{\alpha}^{\intercal}\vb*{\sigma}=0$ (one such $\vb*{\alpha}$ is e.g. given by setting every element except the \mbox{$\bar{h}$-th} to zero, where the \mbox{$\bar{h}$-th} element of $\vb*{\sigma}$ contains the vanishing frequency). With the conditions marked by $\RN{2}$ in $\Omega^{(2)}$, the separated terms $a_{1,1}{\myw}_{1,1}\cos({\myw}_{1,1} t),\allowbreak \dots, \allowbreak a_{1,h-1}{\myw}_{1,h-1}\cos({\myw}_{1,h-1} t), \allowbreak c_{1,h} \bar{\myw}_{1,h} \sin(\bar{\myw}_{1,h} t)$ in \eqref{eq:bew_omega_3} are linearly independent and therefore, their linear combination satisfies \eqref{eq:alpha_sig}, when at least one coefficient is non-zero, i.e. $\vb*{\alpha}\neq\vb*{0}$ holds. Again, in the special case of $\varepsilon_{\bar{h}} = 0$ ($\forall \bar{h} \in \{1,\dots,h\}$) the conditions labeled by $\RN{2}$ in $\Omega^{(2)}$ are also necessary to satisfy \eqref{eq:alpha_sig}. Suppose w.l.o.g. ${\myw}_{1,1} = \pm {\myw}_{1,h-1}$. Then, with $\vb*{\alpha}^{\intercal} = \matc{\mp \frac{1}{a_{1,1}} & 0 & \dots & 0 & \frac{1}{a_{1,h-1}} & 0}$, $\vb*{\alpha}^{\intercal}\vb*{\sigma}=0$ follows.
	Thus, the conditions denoted by $\RN{1}$ and $\RN{2}$ in $\Omega^{(2)}$ are necessary and sufficient to satisfy \eqref{eq:alpha_sig} for any $\vb*{\alpha}\neq\vb*{0}$ when $\varepsilon_{\bar{h}} = 0$ ($\forall \bar{h} \in \{1,\dots,h\}$). The conditions in $\Omega^{(1)}$ and those labeled by $\RN{3}$ in $\Omega^{(2)}$ are sufficient to prevent the cancelation of existing trigonometric functions in case of $\varepsilon_{\bar{h}} \neq 0$. Since the exemplarily chosen column of $\vb*{P}$ is exchangeable and at least one such column exists, any vector $\vb*{\omega} \in \Omega = \Omega^{(1)} \cap \Omega^{(2)}$ leads to an $\vb*{\alpha}^{\intercal}\vb*{\sigma}$ that satisfies \eqref{eq:alpha_sig}. Finally, Proposition~\ref{proposition:def_omega} concludes the PE of $\vb*{\sigma}$.   
\end{proof}

\begin{lemma} \label{lemma:omega_not_empty}
	If Algorithm~\ref{alg:algorithm} does not return $\bar{\Omega}^{(2)} = \mathbb{R}^m$, \eqref{eq:omega_red} follows and $\Omega\neq\emptyset$ holds.
\end{lemma}
\begin{proof}
	In line with Algorithm~\ref{alg:algorithm} and Lemma~\ref{lemma:simpl_omega_2}, if ${\omega_{j}=\omega_{j}}$ ($\forall j \in \{1,\dots,m\}$) is absent in the last iteration (${k_{\myit}\geq N_{\myP}-1}$), $\Omega^{(2)}$ is defined by \eqref{eq:omega_2_red} and \eqref{eq:omega_red} holds for $\Omega$. By requiring that for each row $\vb*{c}^{\intercal}_{z}$ of the coefficient matrices $\vb*{C}_{z}$ the inequality $\vb*{c}^{\intercal}_{z}\vb*{\omega}\neq 0$ has to be fulfilled, a subset $\Omega_{\RN{1}}$ of $\Omega$ follows. 
	Since $\vb*{c}^{\intercal}_{z}\vb*{\omega}=0$ describes an $(m-1)$-dimensional hyperplane in $\mathbb{R}^{m}$, it defines a Lebesgue null set \cite[p.~146]{Cohn.2013}. Additionally, every countable union of Lebesgue null sets is a Lebesgue null set again \cite[Proposition~1.2.4]{Cohn.2013}. Thus, $\Omega_{\RN{1}}$ represents the $\mathbb{R}^{m}$ except for a Lebesgue null set and $\Omega_{\RN{1}} \neq \emptyset$ as well as $\Omega \neq \emptyset$ results.
\end{proof}

\begin{remark}
	If Algorithm~\ref{alg:algorithm} returns $\bar{\Omega}^{(2)} = \mathbb{R}^m$, a promising approach is to introduce more frequency variables in the elements of the desired system states~$\vb*{x}_{\myd}$.
\end{remark}

\begin{lemma}\label{lemma:Np1}
	Let $\Omega$ be as in Theorem~\ref{theorem:frequ_cond}. If $N_{\myP}=1$ holds, $\vb*{\sigma} = \d{\vb*{\phi}}(\vb*{x}_{\myd})$ is PE iff $\vb*{\omega}\in\Omega$.
\end{lemma}
\begin{proof}
	Since $N_{\myP}=1$ results in $\Omega^{(1)} = \mathbb{R}^{m}$ and the non-existence of conditions denoted by $\RN{3}$ in $\Omega^{(2)}$~\eqref{eq:omega_2}, the conditions defining the set $\Omega$ are solely those labeled by $\RN{1}$ and $\RN{2}$ in $\Omega^{(2)}$. Furthermore, $N_{\myP}=1$ corresponds with $\varepsilon_{\bar{h}} = 0$ ($\forall \bar{h} \in \{1,\dots,h\}$) in \eqref{eq:bew_omega_3}. In this case, $\RN{1}$ and $\RN{2}$ in $\Omega^{(2)}$ are necessary and sufficient to guarantee \eqref{eq:alpha_sig} and even $\vb*{\alpha}^{\intercal}\vb*{\sigma} \neq 0$ (cf. proof of Theorem~\ref{theorem:frequ_cond}). Using \cite[Sublemma~6.1]{Narendra.2005} and \cite[Lemma~6.2]{Narendra.2005} as in the proof of Proposition~\ref{proposition:def_omega}, necessity and sufficiency in case of $N_{\myP}=1$ follows.	
\end{proof}

\begin{lemma} \label{lemma:scaling_x_d}
	Consider a set $\Omega$ according to Theorem~\ref{theorem:frequ_cond}. Then, any vector $\vb*{\omega} \in \Omega$ ensures that $\d{\vb*{\phi}}(\vb*{\bar{x}}_{\textup{d}})$ with	
	\begin{align} \label{eq:scaling_x_d}
		\begin{aligned}
			\bar{\vb*{x}}_{\myd} = \textup{diag}\left(\nu_{1}, \nu_{2}, \dots, \nu_{n}\right) \vb*{x}_{\myd}, 
		\end{aligned}
	\end{align}
	$\nu_{j} \in \mathbb{R}_{\neq 0}, \forall j \in \{1,\dots,n\}$, is PE.
\end{lemma}
\begin{proof}
	Inserting \eqref{eq:scaling_x_d} into \eqref{eq:assumption_phi} yields
	\begin{align}
		\phi_{\bar{h}}(\bar{\vb*{x}}_{\myd}) &= \bar{f}_{\bar{h}}\prod_{j=1}^{n} \bar{x}^{f_{j,\bar{h}}}_{\myd,j} \nonumber \\ &= \left( \prod_{j=1}^{n} \nu^{f_{j,\bar{h}}}_{j}\right) \left(\bar{f}_{\bar{h}} \prod_{j=1}^{n} x^{f_{j,\bar{h}}}_{\myd,j}\right) \nonumber \\
		&= \left( \prod_{j=1}^{n} \nu^{f_{j,\bar{h}}}_{j}\right) \phi_{\bar{h}}(\vb*{x}_{\myd}), 
	\end{align}
	$\forall \bar{h} \in \{1,\dots,h\}$. Thus, scaling $\vb*{x}_{\myd}$ in consistency with \eqref{eq:scaling_x_d} solely results in a modification of the coefficients $a_{l,\bar{h}}$, $c_{k,\bar{h}}$ and $e_{\bar{h}}$ in \eqref{eq:phi_x_d_short}. Since their exact values are negligible for the calculation of $\Omega$ and the fulfillment of \eqref{eq:alpha_sig}, any vector $\vb*{\omega} \in \Omega$, where $\Omega$ is calculated on the basis of $\vb*{x}_{\myd}$, ensures that $\d{\vb*{\phi}}(\vb*{\bar{x}}_{\textup{d}})$ is PE.
\end{proof}

In general, the basis function vectors $\vb*{\phi}_{i}(\vb*{x})$ ($i \in \mathcal{N}$) of the $N$ players differ. If, from a given $\vb*{x}_{\myd}$ according to Assumption~\ref{asm:xd_o}, the proposed procedure is done separately for each $\vb*{\phi}_{i}(\vb*{x})$, $N$ sets $\Omega_{i}$ are derived.
Considering their intersection in the same manner as introduced in Lemma~\ref{lemma:intersection} leads to $\Omega = \bigcap_{i=1}^{N} \Omega_{i}$ and any vector $\vb*{\omega} \in \Omega$ ensures that $\d{\vb*{\phi}}_{i}(\vb*{x}_{\myd})$ ($\forall i \in \mathcal{N}$) is PE.

Theorem~\ref{theorem:frequ_cond} is the main result of our paper and provides a novel and generally applicable statement regarding PE signals which are the outcome of nonlinear transformations. It solves Problem~\ref{problem:PE} since $\vb*{x}_{\myd}(t)$ with any $\vb*{\omega} \in \Omega$ represents a suited trajectory $\vb*{x}(t)$.


\section{Excitation Procedures for ADP}\label{sec:excitation_procedures}
Based on the proposed solution to Problem~\ref{problem:PE}, Problem~\ref{problem:excitation} regarding the calculation of excitation signals for (online) ADP algorithms, which aim at solving the differential game defined in Section~\ref{sec:problem} and rely on the PE of $\vb*{\sigma}_i$ ($\forall i \in \mathcal{N}$) to converge to the Nash strategies, can be stated.

\begin{assumption} \label{asm:knwon_fgphi}
	Let $\vb*{f}(\vb*{x})$, $\vb*{g}(\vb*{x})$, $\vb*{\phi}_{i}(\vb*{x})$ and $\vb*{R}_{ii}$ ($\forall i \in \mathcal{N}$) be known. 
\end{assumption}

\begin{problem} \label{problem:excitation}
	Let Assumption~\ref{asm:knwon_fgphi} hold. Compute an excitation signal $\hat{\vb*{u}} = \begin{bmatrix} \hat{\vb*{u}}^{\intercal}_{1} & \hat{\vb*{u}}^{\intercal}_{2} & \dots & \hat{\vb*{u}}^{\intercal}_{N} \end{bmatrix}^{\intercal}$ offline, where ${\hat{\vb*{u}}_{i}: \mathbb{R}_{\geq 0} \rightarrow \mathbb{R}^{p_i}}$ ($i \in \mathcal{N}$) is added to the control law of player~$i$ (${\vb*{u}_{i} = \vb*{\mu}_{i}(\vb*{x}) + \hat{\vb*{u}}_{i}}$), such that $\vb*{\sigma}_i$ ($\forall i \in \mathcal{N}$) is PE.
\end{problem}

Calculating an excitation signal $\hat{\vb*{u}}$ through an inversion-based feed-forward control approach (see e.g. \cite{Fliess.1995}) by using the solution of Problem~\ref{problem:PE}, i.e. a trajectory $\vb*{x}(t)$ guaranteeing that $\vb*{\sigma}_i=\d{\vb*{\phi}}_{i}(\vb*{x}(t))$ is PE $\forall i \in \mathcal{N}$, solves Problem~\ref{problem:excitation}.

First, choose $\vb*{x}_{\myd}(t)$ according to Assumption~\ref{asm:xd_o}. Theorem~\ref{theorem:frequ_cond} leads to a set $\Omega$ such that any $\vb*{\omega} \in \Omega$ ensures that $\dot{\vb*{\phi}}_{i}(\vb*{x}_{\myd})$ ($\forall i \in \mathcal{N}$) is PE. Then, a concrete numerical value $\vb*{\omega}_{\myc} \in \Omega$ and a scaling $\vb*{\nu}$ of $\vb*{x}_{\myd}$ (cf. Lemma~\ref{lemma:scaling_x_d}) leads to the reference trajectory $\bar{\vb*{x}}_{\myd} = \text{diag}(\vb*{\nu})\left.\vb*{x}_{\myd}\right\vert_{\vb*{\omega}=\vb*{\omega}_{\myc}}$ for the inversion-based feed-forward control approach. We briefly discuss the calculation of $\hat{\vb*{u}}$ based on $\bar{\vb*{x}}_{\myd}$ for the example of differentially flat systems \cite{Fliess.1995}.

\begin{assumption} \label{asm:state_linearization}
	Let $\vb*{g}(\vb*{x})$ be rewritten as $\bar{\vb*{g}}(\vb*{x}) = \begin{bmatrix} \vb*{g}_{\RN{1}}(\vb*{x}) & \vb*{g}_{\RN{2}}(\vb*{x}) \end{bmatrix}$, where $\vb*{g}_{\RN{1}}: \mathbb{R}^{n} \rightarrow \mathbb{R}^{n \times p_{\RN{1}}}$, $\vb*{g}_{\RN{2}}: \mathbb{R}^{n} \rightarrow \mathbb{R}^{n \times p_{\RN{2}}}$ and $\vb*{\gamma}(\vb*{x})$ such that $\vb*{g}_{\RN{2}}(\vb*{x}) = \vb*{g}_{\RN{1}}(\vb*{x})\vb*{\gamma}(\vb*{x})$ and $\rank\left( \vb*{g}_{\RN{1}}(\vb*{x}) \right) = p_{\RN{1}}$. Rearrange the system inputs accordingly: $\bar{\vb*{u}} = \begin{bmatrix} \vb*{u}_{\RN{1}}^{\intercal} & \vb*{u}_{\RN{2}}^{\intercal} \end{bmatrix}^{\intercal}$. Assume that for $\d{\vb*{x}} = \vb*{f}_{\myg} (\vb*{x}) + \vb*{g}_{\RN{1}}(\vb*{x}) \vb*{u}_{\RN{1}}$, where $\vb*{f}_{\myg} (\vb*{x}) = \vb*{f}(\vb*{x}) + \vb*{g}(\vb*{x})\vb*{\mu}(\vb*{x})$, an output function $\vb*{y} = \vb*{h}(\vb*{x})$ ($\dim(\vb*{h}(\vb*{x}))= p_{\RN{1}}$) for exact state linearization exists.\footnote{Assumption~\ref{asm:state_linearization} needs to hold on $\mathcal{X} \subseteq \mathbb{R}^{n}$.} 
\end{assumption}

Since $\vb*{g}_{\RN{2}}(\vb*{x})$ can be expressed by $\vb*{g}_{\RN{1}}(\vb*{x})$, $\hat{\vb*{u}}_{\RN{2}}$ is set to zero without loss for the excitation. The remaining part $\hat{\vb*{u}}_{\RN{1}}$ of the excitation signal $\hat{\bar{\vb*{u}}}$ is derived analytically. Due to the relation between the exact state linearizability of a system and its flatness \cite{Fliess.1995}, from Assumption~\ref{asm:state_linearization} the flatness of $\d{\vb*{x}} = \vb*{f}_{\myg} (\vb*{x}) + \vb*{g}_{\RN{1}}(\vb*{x}) \vb*{u}_{\RN{1}}$ follows with $\vb*{y}$ as flat output. Moreover, the output function $\vb*{h}(\vb*{x})$ leads to full differential order and the diffeomorphism $\vb*{t}(\vb*{x})$ transforming the considered system into nonlinear controllable canonical form can be stated \cite[p.~245]{Isidori.1989}). From this normal form 
\begin{align}
	\vb*{z} &= \vb*{\Psi}_{1}\left( \vb*{y}, \d{\vb*{y}}, \dots, \vb*{y}^{(\delta_{\max} - 1)}  \right)
	\intertext{and}
	\vb*{u}_{\RN{1}} &= \vb*{\Psi}_{2}\left( \vb*{y}, \d{\vb*{y}}, \dots, \vb*{y}^{(\delta_{\max})} \right),
\end{align}
where $\vb*{z}$ is the transformed system state and $\delta_{\max} = \max\{\delta_1, \dots, \delta_{p_{\RN{1}}}\}$ ($\delta_j$ represent the relative degrees \cite[p.~235]{Isidori.1989}), are derived. 
With the chosen $\bar{\vb*{x}}_{\myd}$, $\hat{\vb*{u}}_{\RN{1}}$ follows from $\vb*{\Psi}_{2}$.
Rearranging $\hat{\bar{\vb*{u}}}$ results in $\hat{\vb*{u}}$, which is in general a function of $t$ and $\vb*{\theta}_i$.

\begin{assumption} \label{asm:excitation_regulation}
	Let $\hat{\vb*{u}}$ be the excitation signal achieved by the flatness-based feed-forward control approach based on Assumption~\ref{asm:state_linearization}. Assume that the system state $\hat{\vb*{x}}$, which is calculated a priori by using $\vb*{\Psi}_{1}$, $\vb*{t}(\vb*{x})$ and $\bar{\vb*{x}}_{\myd}$,\footnote{Thus, $\hat{\vb*{x}}$ denotes the system state without online influences such as transient phenomena.} results in
	\begin{align} \label{eq:asm_sigma_hat}
		\begin{aligned}
			\hat{\vb*{\sigma}}_{i}\coloneqq{\vb*{\sigma}}_{i}(\hat{\vb*{x}}(t)) = \d{\vb*{\phi}}_{i}(\hat{\vb*{x}}(t)) = \vb*{M}_{i} \dv{\myi(t)}{t}, 
		\end{aligned}
	\end{align}
	where $\vb*{M}_i \in \mathbb{R}^{h_i \times \bar{h}_i}$ ($h_i \leq \bar{h}_i$) and $\myi$ denotes a vector with an orthogonal functional system of sines and cosines in its elements\footnote{Whether \eqref{eq:asm_sigma_hat} holds, is a priori verifiable. The frequencies in $\myi$ can be expressed as linear combinations of the entries in $\vb*{\omega}_{\myc}$.}. Furthermore, suppose that the system state $\vb*{x}$ resulting online from the controllers $\vb*{\mu}_1,\dots,\vb*{\mu}_N$ and the excitation signal $\hat{\vb*{u}}$ leads to 
	\begin{align}\label{eq:sigma_hat+epsilon}
		\vb*{\sigma}_{i}(t)=\hat{\vb*{\sigma}}_{i}(t)+\vb*{\varepsilon}_{1,i}(t)+\vb*{\varepsilon}_{2,i}(t), \quad \forall i \in \mathcal{N},
	\end{align}
	where $\vb*{\varepsilon}_{1,i}, \vb*{\varepsilon}_{2,i}: \mathbb{R}_{\geq 0} \rightarrow \mathbb{R}^{h_i}$ are continuous and bounded and fulfill $\norm{\vb*{\varepsilon}_{1,i}(t)}\leq \bar{\varepsilon}_{\sigma,i}$, $\forall t$ and $\vb*{\varepsilon}_{2,i}(t) \rightarrow \vb*{0}$ ($t \rightarrow \infty$).
\end{assumption}

\begin{lemma} \label{lemma:excitation_regulation}
	Let $\bar{\vb*{x}}_{\myd}(t) = \diag(\vb*{\nu})\left.\vb*{x}_{\myd}(t)\right\vert_{\vb*{\omega}=\vb*{\omega}_{\myc}}$ follow from Lemma~\ref{lemma:scaling_x_d} where values $\vb*{\omega}_{\myc} \in \Omega$ and $\vb*{\nu}$ have been chosen. Under Assumption~\ref{asm:excitation_regulation}, $\vb*{\sigma}_i (t)$ ($\forall i \in \mathcal{N}$) is PE if $\bar{\varepsilon}_{\sigma,i}$ is sufficiently small and $\rank(\vb*{M}_i)=h_i$ $\forall i \in \mathcal{N}$.
\end{lemma}
\begin{proof}
	The proof is given in Appendix~\ref{app:proof_lemma_excitation_regulation}.
\end{proof}

Assumption~\ref{asm:excitation_regulation} takes into account that inversion-based feed-forward control goes along with transient phenoma. Whereas this is considered by $\vb*{\varepsilon}_{2,i}(t)$ in \eqref{eq:sigma_hat+epsilon}, $\vb*{\varepsilon}_{1,i}(t)=\vb*{0}$ holds under Assumption~\ref{asm:knwon_fgphi}.

\begin{remark} \label{remark:dofs}
	The degrees of freedom (DOFs), namely ${\vb*{\omega}_{\myc} \in \Omega}$ and $\vb*{\nu}$, provided by the proposed excitation procedure can be used to adjust the excitation to further constraints. This includes application-specific requirements, such as low-pass characteristics, or a better fulfillment of the PE conditions, such as the degree of PE\footnote{The lower bounds $\alpha$ and $\alpha_{\RN{1}}$ in \eqref{eq:pe_cond_1} and \eqref{eq:pe_cond_2} are called degrees of PE.}.
\end{remark}

\begin{remark} \label{remark:tracking}
	For tracking controllers, $\bar{\vb*{x}}_{\myd}$ can directly be used as sufficiently rich reference trajectory. Hereto, a sufficient tracking performance in the absence of PE needs to be assumed\footnote{Realizing sufficient tracking performance under the absence of PE, improving performance and guaranteeing parameter convergence with PE afterwards is a common approach in the adaptive control literature (see e.g. \cite{Adetola.2008}).} such that \eqref{eq:sigma_hat+epsilon} holds with $\hat{\vb*{\sigma}}_i(t) \equiv \dot{\vb*{\phi}}_i(\bar{\vb*{x}}_{\myd}(t))$ and $\bar{\varepsilon}_{\sigma,i}$ sufficiently small. From the first part of the proof of Lemma~\ref{lemma:excitation_regulation} the PE of $\vb*{\sigma}_{i}$ ($\forall i \in \mathcal{N}$) results. While $\vb*{\varepsilon}_{1,i}(t)$ describes small tracking errors in this case, $\vb*{\varepsilon}_{2,i}(t)$ represents transient phenomena again.
\end{remark}

\begin{remark}
	Although $\vb*{\phi}_i(\vb*{x})$ is introduced as basis function vector for value function approximation in \eqref{eq:approx_value_i}, it can also represent an arbitrary regressor used by the ADP algorithm considered. The sufficient excitation conditions in Section~\ref{sec:excitation_conditions} yield state trajectories $\vb*{x}_{\myd}$ shaping this regressor (and its time derivative) PE. If Assumption~\ref{asm:phi_h} is not fulfilled, Taylor approximations might be used. In case of an inversion-based feed-forward excitation design, Lemma~\ref{lemma:excitation_regulation} can be used to show that the signal remains PE under small approximation errors.
\end{remark}

The next lemma introduces eigenvalue signals that can be used to verify the fulfillment of the PE condition in simulation. Here, $\lambda_{\min}(\cdot)$ denotes the minimal eigenvalue of a matrix.
\begin{lemma} \label{lemma:eigenvalue_signals}
	Let $\vb*{\sigma}: \mathbb{R}_{\geq 0} \rightarrow \mathbb{R}^{h}$. Then, $\vb*{\sigma}$ is PE $\forall t \geq t_{0}$ iff constants $\alpha > 0$ and $T > 0$ exist such that 
	\begin{align} \label{eq:eigvalue_signal_1}
		\begin{aligned}
			\lambda_{1}(t) = \lambda_{\min}\left( \int_{t}^{t+T} \vb*{\sigma}\vb*{\sigma}^{\intercal} \,\mathrm{d}\tau \right) \geq \alpha > 0, \quad \forall t \geq t_{0}
		\end{aligned}
	\end{align} 
	holds. Furthermore, if $\vb*{\sigma}$ is PE,
	\begin{align} \label{eq:eigvalue_signal_2}
		\begin{aligned}
			\lambda_{2}(t) = \lambda_{\min}\left( \int_{t_{0}}^{t} \vb*{\sigma}\vb*{\sigma}^{\intercal} \,\mathrm{d}\tau \right), \quad \forall t \geq t_{0}
		\end{aligned}
	\end{align}
	increases monotonically and $\lambda_{2}(kT + t_{0}) \geq k\alpha$ ($k \in \mathbb{N}_{\geq 0}$).
\end{lemma}
\begin{proof}
	The proof is given in Appendix~\ref{app:proof_lemma_eigenvalue}.
\end{proof}

Since the calculation of $\lambda_{1}(t)$ requires a parameter $T$, we first compute $\lambda_{2}(t)$ and identify the point in time $T_{1}$, where $\lambda_{2} \geq \alpha_{1}$ holds ($\alpha_{1}$ is a numerical threshold). Using this $T_{1}$ to compute $\lambda_{1}(t)$, one can verify if \eqref{eq:eigvalue_signal_1} holds with $T=T_{1}$ and $\alpha=\alpha_{1}$. Then, according to Lemma~\ref{lemma:eigenvalue_signals}, $\vb*{\sigma}$ is PE.


\section{Numerical Example} \label{sec:simulation}
The 2-player test scenario is defined by the system functions
\begin{align} 
	\vb*{f}(\vb*{x}) \!&=\!\! \begin{bmatrix} -2x_{1} \!+\! x_{2} \\ -x_{2} \!\!-\!\! \frac{1}{2}x_{1} \!\!+\!\! \frac{1}{4} x_{2}\left( (\cos(2x_{1})\!+\!2)^{2} \!+\! (\sin(4x^{2}_{1})\!+\!2)^{2} \right)\end{bmatrix}\!\!, \label{eq:example_f} \\
	\vb*{g}(\vb*{x}) \!&=\!\! \begin{bmatrix} \vb*{g}_1(\vb*{x}) & \!\!\!\vb*{g}_2(\vb*{x}) \end{bmatrix}\!=\!\!\begin{bmatrix} 0 & \!\!\!0 \\ \cos(2x_{1}) \!+\! 2 & \!\!\!\sin(4x^{2}_{1}) \!+\! 2 \end{bmatrix} \label{eq:example_g}
\end{align}
and the value functions of the two players
\begin{align} \label{eq:example_J}
	\begin{aligned}
		V^{\vb*{\mu}}_{1}(\vb*{x}(t)) &= \int_{t}^{\infty} 2(x^{2}_{1} + x^{2}_{2}) + 2(\mu^{2}_{1} + \mu^{2}_{2}) \,\mathrm{d}\tau, \\
		V^{\vb*{\mu}}_{2}(\vb*{x}(t)) &= \frac{1}{2}V^{\vb*{\mu}}_{1}(\vb*{x}(t)).
	\end{aligned}
\end{align}
$V^{*}_{1}(\vb*{x}) = \frac{1}{2} x^{2}_{1} + x^{2}_{2}$ and $V^{*}_{2}(\vb*{x}) = \frac{1}{4} x^{2}_{1} + \frac{1}{2} x^{2}_{2}$ result (see \cite{Nevistic.1996}). Let $\vb*{\phi}(\vb*{x}) = \vb*{\phi}_{1}(\vb*{x}) = \vb*{\phi}_{2}(\vb*{x}) = \begin{bmatrix} x^{2}_{1} & x_{1}x_{2} & x^{2}_{2} \end{bmatrix}^{\intercal}$. The Nash strategies are calculated online by a real-time implementation of a classical policy iteration (PI) algorithm. The players start with initial weights $\vb*{\theta}_{1,0} = \begin{bmatrix} 1.783 & -2.33 & 2.215 \end{bmatrix}^{\intercal}$ and $\vb*{\theta}_{2,0} = \begin{bmatrix} 0.8916 & -1.165 & 1.107 \end{bmatrix}^{\intercal}$, fix their corresponding control laws ($\vb*{\mu}^{0}_{1/2}$) and use a gradient-based adaption procedure to improve their value function approximations. After the policy evaluation is completed ($\vb*{\theta}_{1/2} \rightarrow \vb*{\theta}^{\vb*{\mu}^{0}}_{1/2}$, determined in simulation through $\norm{\vb*{\theta}_{1/2}(t-\tau)-\vb*{\theta}_{1/2}(t)}_{2} < 10^{-3}$ where $\tau$ denotes a design parameter chosen as $\tau = 40\,\text{s}$), a synchronous policy improvement follows ($\vb*{\mu}^{1}_{1/2}$ based on \eqref{eq:approx_mu_i} and $\vb*{\theta}^{\vb*{\mu}^{0}}_{1/2}$). Then, the next iteration of the PI algorithm starts with the evaluation of the updated policies. This iterative procedure aims at achieving the optimal weights $\vb*{\theta}^{*}_{1} = \begin{bmatrix} 0.5 & 0 & 1 \end{bmatrix}^{\intercal}$ and $\vb*{\theta}^{*}_{2} = \begin{bmatrix} 0.25 & 0 & 0.5 \end{bmatrix}^{\intercal}$. The PE of $\vb*{\sigma}=\vb*{\sigma}_{1}=\vb*{\sigma}_{2}$ is necessary and sufficient for convergence in the policy evaluation step\footnote{In order to ensure that $V^{\vb*{\mu}}_i$ is learned during the policy evaluation step, we realize an off-policy procedure indirectly by assuming that the controllers calculate $\vb*{\sigma}=\left( \pdv{\vb*{\phi}(\vb*{x})}{\vb*{x}} \right)^\intercal \left( \vb*{f} + \vb*{g}\vb*{\mu} \right)$ from the current state $\vb*{x}$.}. Finally, the convergence to the optimal weights follows from the convergence of the PI algorithm and the uniformly ultimately boundedness of the system states results under admissible initial policies. Since we are using a real-time implementation of a classical PI algorithm, this convergence and stability behavior results from the convergence behavior of the PI algorithm itself (see e.g. \cite[Theorem~1]{Liu.2014} or \cite[Theorem~2]{Song.2017}).

The analytical calculations to derive our excitation signals are performed in Maple. By choosing $\vb*{x}_{\myd} = \begin{bmatrix} \sin(\omega_{1}t)+\sin(\omega_{2}t) & \sin(\omega_{3}t) \end{bmatrix}^{\intercal}$, the set $\Omega$ is defined through 49 excluding conditions on the frequency variables $\vb*{\omega} \in \mathbb{R}^{3}$, e.g. $\matc{1&0&0}\vb*{\omega}\neq 0$. Selecting the DOFs leads to the three exemplary system state trajectories: 
\begin{align} \label{eq:example_xd}
	\begin{aligned}
		\bar{\vb*{x}}_{\myd,1} &= 0.25\begin{bmatrix} \sin(1\,\frac{1}{\text{s}}\!\cdot \! t)\!+\!\sin(2\,\frac{1}{\text{s}}\cdot t) & \!\!\!\sin(3\,\frac{1}{\text{s}}\!\cdot \! t) \end{bmatrix}^{\intercal}, \\
		\bar{\vb*{x}}_{\myd,2} &= 0.25\begin{bmatrix} \sin(0.5\,\frac{1}{\text{s}} \!\cdot \! t)\!+\!\sin(1\,\frac{1}{\text{s}}\! \cdot \! t) & \!\!\!\sin(2\,\frac{1}{\text{s}} \cdot \! t) \end{bmatrix}^{\intercal}, \\
		\bar{\vb*{x}}_{\myd,3} &= 2\bar{\vb*{x}}_{\myd,2}.
	\end{aligned}
\end{align}
Since $\vb*{g}_{2}(\vb*{x})$ in~\eqref{eq:example_g} can be expressed in terms of $\vb*{g}_{1}(\vb*{x})$, let $\vb*{g}_{\RN{1}}(\vb*{x}) = \vb*{g}_{1}(\vb*{x})$ and $\vb*{g}_{\RN{2}}(\vb*{x}) = \vb*{g}_{2}(\vb*{x})$. Thus, $\hat{u}_{2} = \hat{u}_{\RN{2}}$ is set to zero for the excitation signal $\hat{\vb*{u}} = \begin{bmatrix} \hat{u}_{1} & \hat{u}_{2} \end{bmatrix}^{\intercal}$. By considering $\vb*{g}_{1}(\vb*{x})$, $h(\vb*{x})=x_{1}$ can be deduced as an output function for exact state linearization (cf. Assumption~\ref{asm:state_linearization}). With the flatness-based feed-forward control approach, three excitation signals $\hat{u}_{1,j}(t,\vb*{\theta}_1,\vb*{\theta}_2)$ ($j \in \{1,2,3\}$) are computed from $\bar{\vb*{x}}_{\myd,j}$~\eqref{eq:example_xd}. The weights are chosen online from the control laws currently used by the players. Considering Lemma~\ref{lemma:excitation_regulation}, $\hat{\vb*{\sigma}} = \hat{\vb*{\sigma}}_{1} = \hat{\vb*{\sigma}}_{2}$ can be expressed through \eqref{eq:asm_sigma_hat} and $\vb*{M}_1 = \vb*{M}_2 \in \mathbb{R}^{3 \times 8}$ has full row rank. White Gaussian probing noise $\hat{u}_{1,\text{n}}$ is included in the comparison as state of the art reference. To suppress the influence of $\eta_{i}$ and different amplitudes of $\vb*{\sigma}_{i}$, normalizing learning rates $\eta_{i} = \frac{\eta_{\text{v}}}{\sigma^{2}_{i,\max}}$, $\eta_{\text{v}} \gg 1$, are used $\forall i \in \{1,2\}$. Thus, $\rho_{i} \approx 1 - \frac{1}{\eta_{\text{v}}}\frac{2}{T_{i}}\frac{\alpha^{2}_{\RN{1},i}}{\sigma^{2}_{i,\text{max}}}$ follows from \eqref{eq:rho}. Hence, a reduction of $\rho_{i}$ (i.e. faster convergence) can only be achieved by a better fulfillment of the PE inequalities, i.e. higher degrees $\alpha_{\RN{1},i}$ under the same upper boundaries $\sigma_{i,\text{max}}$ or lower time constants $T_{i}$. With $\eta_{\text{v}} = 10000$, the learning rates are: $\hat{u}_{1,\text{n}}$: $\eta_{i} = \frac{\eta_{\text{v}}}{21}$, $\hat{u}_{1,1}$: $\eta_{i} = \frac{\eta_{\text{v}}}{274}$, $\hat{u}_{1,2}$: $\eta_{i}= \frac{\eta_{\text{v}}}{125}$ and $\hat{u}_{1,3}$: $\eta_{i} = \frac{\eta_{\text{v}}}{1670}$, $\forall i \in \{1,2\}$.
 
Fig.~\ref{fig:lambda_min_1} proves the PE of $\vb*{\sigma}$ (cf. Lemma~\ref{lemma:eigenvalue_signals}) with the analytical excitation signals and with the probing noise. Fig.~\ref{fig:error_norm} visualizes for player~2 (player~1 shows similar trends) that all excitation signals achieve the desired behavior $||\tilde{\vb*{\theta}}_{2}||_{2} \rightarrow 0$. The analytical excitation signals outperform the probing noise significantly with a convergence time reduction of $92.7\,\%$ for the best performing signal $\hat{u}_{1,1}(t)$ compared to the state of the art reference. Due to the normalizing learning rates this improvement is traced back to a better fulfillment of the PE conditions and scaling $\vb*{x}_{\myd}$ does not influence the convergence times noticeably (see results of $\hat{u}_{1,2}(t)$ and $\hat{u}_{1,3}(t)$).

\begin{figure}[t]
	\centering
	\includegraphics[width=3.5in]{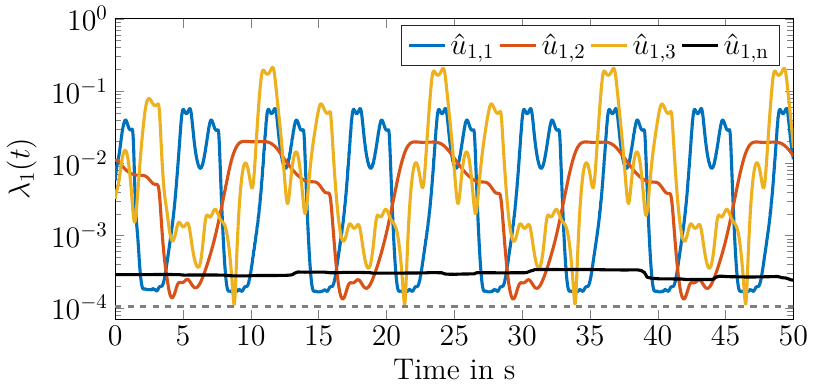}
	\caption{Eigenvalue signal $\lambda_{1}(t)$ calculated with $T_{1}$ ($\hat{u}_{1,\text{n}}$: $T_{1}=69.91\,\text{s}$, $\hat{u}_{1,1}$: $T_{1}=1.14\,\text{s}$, $\hat{u}_{1,2}$: $T_{1}=3.61\,\text{s}$, $\hat{u}_{1,3}$: $T_{1}=1.54\,\text{s}$). The dashed horizontal line corresponds with the numerical threshold value $\alpha_{1}=10^{-4}$.}
	\label{fig:lambda_min_1}
\end{figure}

\begin{figure}[t]
	\centering
	\includegraphics[width=3.5in]{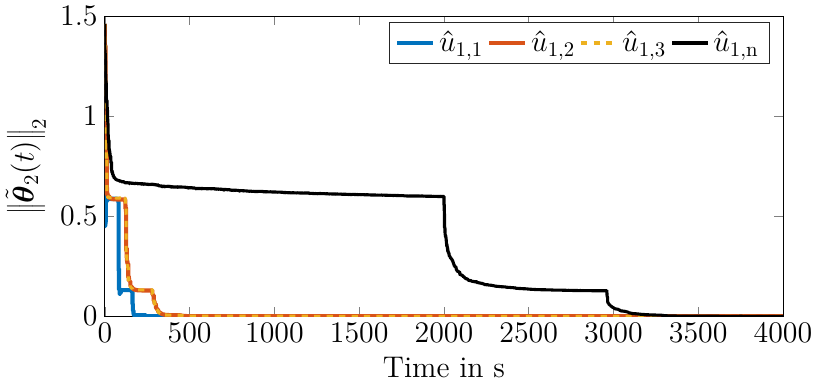}
	\caption{Weight error norm of player 2. The convergence times (${||\tilde{\vb*{\theta}}_2(t)||_2 < 10^{-3}}$) are: $\hat{u}_{1,\text{n}}$: $3392\,\text{s}$, $\hat{u}_{1,1}$: $247\,\text{s}$, $\hat{u}_{1,2}$: $482\,\text{s}$, $\hat{u}_{1,3}$: $477\,\text{s}$.}
	\label{fig:error_norm}
\end{figure}


\section{Conclusion} \label{sec:conclusion}
In this paper, we propose sufficient conditions under which a polynomial transformation of a signal is PE. These conditions are then used for the analytical calculation of excitation signals in ADP while providing degrees of freedom. Thus, our exemplary excitation procedure based on our conditions for PE provides capabilities to design application-specific excitation. The numerical example shows that all analytically calculated excitation signals ensure the fulfillment of the PE conditions along with outperforming state of the art approaches represented by white Gaussian probing noise. The derived signals are able to reduce the convergence time to the optimal behavior of an ADP controller by up to $92.7\,\%$ in this example.


\appendices

\section{Proof of Lemma~\ref{lemma:trig_simplification}}\label{app:proof_lemma_trig_simplification}
\begin{proof}
	Inserting one element $x_{\myd,o}$ of $\vb*{x}_{\myd}$~\eqref{eq:assumption_x_d} in one factor $x^{f_{o,\bar{h}}}_{o}$ of $\phi_{\bar{h}}(\vb*{x})$~\eqref{eq:assumption_phi} and applying the multinomial theorem, power-reduction formulae and product-to-sum identities yields
	\begin{align}\label{eq:bew_trig_simp_1}
		x^{f_{o,\bar{h}}}_{\myd,o} &= \left( \sum_{j=1}^{m} g_{j,o}\sin(\omega_{j}t) + \sum_{j=1}^{m} \bar{g}_{j,o}\cos(\omega_{j}t) \right)^{f_{o,\bar{h}}} \nonumber \\
		&= \sum_{l=1}^{L^{(o)}_{\bar{h}}} a^{(o)}_{l,\bar{h}} \sin(\sum_{j=1}^{m} b^{(o)}_{j,l,\bar{h}} \omega_{j} t) \nonumber \\
		&\phantom{=}+ \sum_{k=1}^{K^{(o)}_{\bar{h}}} c^{(o)}_{k,\bar{h}} \cos(\sum_{j=1}^{m} d^{(o)}_{j,k,\bar{h}} \omega_{j} t) + e^{(o)}_{\bar{h}}. 
	\end{align}
	In \eqref{eq:bew_trig_simp_1}, the indices $\bar{h}$ and $(o)$ of the parameters $a^{(o)}_{l,\bar{h}}, c^{(o)}_{k,\bar{h}} \in \mathbb{R}_{\neq 0}$, $e^{(o)}_{\bar{h}}, b^{(o)}_{j,l,\bar{h}}, d^{(o)}_{j,k,\bar{h}} \in \mathbb{R}$ characterize their dependency on the exact element $\phi_{\bar{h}}(\vb*{x}_{\myd})$ and factor $x^{f_{o,\bar{h}}}_{\myd,o}$. In order to state $a^{(o)}_{l,\bar{h}}, c^{(o)}_{k,\bar{h}} \neq 0$, denote the upper sum limits $L^{(o)}_{\bar{h}}, K^{(o)}_{\bar{h}} \in \mathbb{N}_{\geq 0}$ with the same dependencies.	
	Because \eqref{eq:bew_trig_simp_1} holds $\forall o \in \{1,\dots,n\}$, for each element
	\begin{align} \label{eq:phi_x_d}
		\phi_{\bar{h}}(\vb*{x}_{\myd}) &= \bar{f}_{\bar{h}}\prod_{o=1}^{n} x^{f_{o,\bar{h}}}_{\myd,o} = \sum_{l=1}^{L_{\bar{h}}} a_{l,\bar{h}} \sin(\sum_{j=1}^{m} b_{j,l,\bar{h}} \omega_{j} t) \nonumber \\
		&\hphantom{=}+ \sum_{k=1}^{K_{\bar{h}}} c_{k,\bar{h}} \cos(\sum_{j=1}^{m} d_{j,k,\bar{h}} \omega_{j} t) + e_{\bar{h}}
	\end{align}
	follows. Eq.~\eqref{eq:phi_x_d} results from expanding the product of the sums~\eqref{eq:bew_trig_simp_1} and applying product-to-sum identities again. As a result of the requirements, $f_{o,\bar{h}} \neq 0$ for at least one $o$, $m \geq 1$, $g_{j,o} \neq 0$ or $\bar{g}_{j,o} \neq 0$ for at least one $j$, in Assumption~\ref{asm:phi_h} and Assumption~\ref{asm:xd_o} it follows that $L_{\bar{h}}+K_{\bar{h}}\geq1$ ($\forall \bar{h} \in \{1,\dots,h\}$). Thus, \eqref{eq:phi_x_d} reduces to \eqref{eq:phi_x_d_short}.
\end{proof}

\section{Proof of Lemma~\ref{lemma:excitation_regulation}}\label{app:proof_lemma_excitation_regulation}
\begin{proof}
	The proof is divided into two parts. First, we show that $\vb*{\sigma}_i$ is PE if $\hat{\vb*{\sigma}}_i$ is PE. Second, we prove the PE of $\hat{\vb*{\sigma}}_i$.  
	
	With \eqref{eq:sigma_hat+epsilon}, \eqref{eq:pe_cond_2} yields
	\begin{align} 
		&\frac{1}{T}\int_{t}^{t+T} \left|\left(\hat{\vb*{\sigma}}_{i}+\vb*{\varepsilon}_{1,i}+\vb*{\varepsilon}_{2,i}\right)^{\intercal}\vb*{\iota}\right| \,\mathrm{d}\tau \geq \frac{1}{T}\int_{t}^{t+T} \abs{\hat{\vb*{\sigma}}_{i}^{\intercal}\vb*{\iota}} \,\mathrm{d}\tau \nonumber \\
		&- \frac{1}{T}\int_{t}^{t+T} \abs{\vb*{\varepsilon}_{1,i}^{\intercal}\vb*{\iota}} \,\mathrm{d}\tau - \frac{1}{T}\int_{t}^{t+T} \abs{\vb*{\varepsilon}_{2,i}^{\intercal}\vb*{\iota}} \,\mathrm{d}\tau \nonumber \\ 
		&\geq \alpha_{\RN{1},i} - \bar{\varepsilon}_{\mysigma,i} - \frac{1}{T} M_{\varepsilon}(T).\label{eq:bew_excitation_track_1}
	\end{align}
	The last inequality in \eqref{eq:bew_excitation_track_1} follows from the PE of $\hat{\vb*{\sigma}}_{i}$, the boundedness of $\vb*{\varepsilon}_{1,i}$ and the convergence behavior of $\vb*{\varepsilon}_{2,i}$, which leads to
	\begin{align} \label{eq:bew_excitation_track_2}
		\int_{t}^{t+T} \abs{\vb*{\varepsilon}_{2,i}^{\intercal}\vb*{\iota}} \,\mathrm{d}\tau \leq M_{\varepsilon}(T),
	\end{align}
	where $M_{\varepsilon}(T)$ is an upper bound dependent on $T$. Since $M_{\varepsilon}(T)$ saturates with increasing $T$, a $\bar{T}$ exists such that $\alpha_{\RN{1},i} > \bar{\varepsilon}_{\mysigma,i} + \frac{1}{T} M_{\varepsilon}(T)$ ($\forall T > \bar{T}$) if $\bar{\varepsilon}_{\mysigma,i}$ is sufficiently small. According to \eqref{eq:pe_cond_2}, the PE of $\vb*{\sigma}_{i}$ results under the assumed PE of $\hat{\vb*{\sigma}}_{i}$. PE of $\hat{\vb*{\sigma}}_{i}$ follows from \eqref{eq:asm_sigma_hat} using \cite[Lemma~6.1]{Narendra.2005} since the temporal derivative of $\myi$ is PE and $\vb*{M}_i$ has full row rank.
\end{proof}

\section{Proof of Lemma~\ref{lemma:eigenvalue_signals}}\label{app:proof_lemma_eigenvalue}
\begin{proof}
	Definition~\ref{def:pe_new} yields
	\begin{align} \label{eq:bew_eigvalue_1}
		\begin{aligned}
			\vb*{\Xi}_{1}(t) \coloneqq \int_{t}^{t+T} \vb*{\sigma}\vb*{\sigma}^{\intercal} \,\mathrm{d}\tau \succeq \alpha\vb*{I}. 
		\end{aligned}
	\end{align}
	Due to the symmetry of $\vb*{\Xi}_{1}(t)$, \eqref{eq:bew_eigvalue_1} leads to
	\begin{align} \label{eq:bew_eigvalue_2}
		\begin{aligned}
			&\vb*{\Xi}_{1}(t) = \vb*{M}(t)\vb*{G}(t)\vb*{M}^{-1}(t) \succeq \alpha\vb*{I} \\
			&\Rightarrow \lambda_{1}(t) = \lambda_{\min}(\vb*{\Xi}_{1}(t)) = \vb*{\iota}^{\intercal}(t)\vb*{G}(t)\vb*{\iota}(t) \geq \alpha,
		\end{aligned}
	\end{align}
	where $\vb*{G}(t)$ is a diagonal matrix containing the eigenvalues of $\vb*{\Xi}_{1}(t)$, $\vb*{M}(t)$ a regular matrix for diagonalizing $\vb*{\Xi}_{1}(t)$ and $\vb*{\iota}(t)$ the unit vector extracting the minimal eigenvalue of $\vb*{\Xi}_{1}(t)$.
	
	Suppose $\vb*{\sigma}$ is PE. Thus, from Definition~\ref{def:pe_new}
	\begin{align} \label{eq:bew_eigvalue_4}
		\begin{aligned}
			\vb*{\Xi}_{2}(kT + t_{0}) \coloneqq \sum_{j=0}^{k-1} \underbrace{\int_{jT + t_{0}}^{(j+1)T + t_{0}} \vb*{\sigma}\vb*{\sigma}^{\intercal} \,\mathrm{d}\tau}_{\succeq \alpha\vb*{I}} \succeq k\alpha \vb*{I}
		\end{aligned}
	\end{align}
	results. By analogy to \eqref{eq:bew_eigvalue_2}, from \eqref{eq:bew_eigvalue_4}
	\begin{align} \label{eq:bew_eigvalue_5}
		\begin{aligned}
			\lambda_{2}(kT + t_{0}) = \lambda_{\min}(\vb*{\Xi}_{2}(kT + t_{0})) \geq k\alpha
		\end{aligned}
	\end{align}
	follows. Furthermore, since for $kT + t_{0} < t < (k+1)T + t_{0}$
	\begin{align} \label{eq:bew_eigvalue_6}
		\begin{aligned}
			\vb*{\Xi}_{2}(t) = \vb*{\Xi}_{2}(kT + t_{0}) + \underbrace{\int_{kT + t_{0}}^{t} \vb*{\sigma}\vb*{\sigma}^{\intercal} \,\mathrm{d}\tau}_{\succeq \vb*{0}} \succeq \vb*{\Xi}_{2}(kT)
		\end{aligned}
	\end{align}
	holds, the increase of \eqref{eq:eigvalue_signal_2} is monotonic.
\end{proof}



\end{document}